\documentclass[12pt]{article}
\usepackage{geometry}
\usepackage{mathtools}
\mathtoolsset{showonlyrefs}
\usepackage{float}
\usepackage{amsmath,amssymb,amsthm}
\usepackage[american]{babel}
\usepackage[latin1]{inputenc}
\usepackage[numbers]{natbib}
\usepackage{mathrsfs}
\usepackage{booktabs}
\usepackage{url}
\usepackage{amscd}
\usepackage{amstext}
\usepackage{amsfonts}
\usepackage{bbm}
\usepackage{soul}
\usepackage[latin1]{inputenc}
\usepackage[T1]{fontenc}
\usepackage{enumerate}
\usepackage{color,graphicx}
\numberwithin{equation}{section}
\usepackage{todonotes}

\newcommand{\mathset}[1]{\mathbbm{#1}}
\newcommand{\morf}[4][\to]{ #2 \colon #3 #1 #4}
\newcommand{\F}{\mathcal{F}}

\newcommand{\Fi}{\mathbb{F}}

\newcommand{\V}{\mathcal{V}}

\newcommand{\A}{\textbf{A}}

\newcommand{\s}{\mathscr S}

\newcommand{\B}{\mathscr B}
\newcommand{\N}{\mathset{N}}

\newcommand{\ca}{c\`adl\`ag}
\newcommand{\R}{\mathset{R}}

\newcommand{\abs}[2][]{#1\lvert #2#1\rvert}

\newcommand{\p}[2][]{#1 ( #2_t #1)_{t\geq 0}}

\renewcommand{\B}{\mathscr B}

 \def \lt{[\hskip-1.3pt[ }
\def \rt{ ]\hskip-1.3pt]}

\def\E{\mathbb{E}}
\def\P{\mathbb{P}}
\def\1{\mathbf{1}}
\def\({\left(}
\def\){\right)}
\def\np{\par\noindent}
\def\X{\mathbf{X}}
\def\M{\mathbf{M}}
\def\Y{\mathbf{Y}}

\def\citemma{\cite[Proposition~3.5]{fv-mma}}

\theoremstyle{plain}
\newtheorem{lemma}{Lemma}[section]
\newtheorem{proposition}[lemma]{Proposition}
\newtheorem{theorem}[lemma]{Theorem}
\newtheorem{corollary}[lemma]{Corollary}
\newtheorem*{theorem*}{Theorem}

\theoremstyle{definition}
\newtheorem{example}[lemma]{Example}

\theoremstyle{definition}
\newtheorem{definition}[lemma]{Definition}
\newtheorem{remark}[lemma]{Remark}

\newcommand{\devnull}[1]{}

\title{On infinitely divisible semimartingales}
\author{Andreas Basse-O'Connor$^{\dagger}$ and Jan Rosi\'nski$^{\ddagger}$\\
{\normalsize Aarhus University and University of Tennessee }\\
{\normalsize $^\dagger$E-mail: basse@imf.au.dk\qquad $^\ddagger$E-mail: rosinski@math.utk.edu}
}
\date{December 10, 2014}

\begin{document}

\maketitle

\begin{abstract}
Stricker's theorem states that a Gaussian process is a semimartingale in its natural filtration if and only if it is the sum of an independent increment Gaussian process  and a Gaussian process of finite variation, see [1983, \emph{Z. Wahrsch. Verw. Gebiete~64}(3)]. We consider extensions of this result to non Gaussian infinitely  divisible processes. First we show that the class of infinitely divisible semimartingales is so large that the natural analog of Stricker's theorem fails to hold. Then, as the main result, we prove that an infinitely divisible semimartingale relative to the filtration generated by a random measure admits a unique decomposition into an independent increment process and an infinitely divisible process of finite variation. Consequently, the natural analog of Stricker's theorem holds for  all strictly representable processes (as defined in this paper). Since Gaussian processes are strictly representable due to Hida's multiplicity theorem, the classical Stricker's theorem follows from our result. 
Another consequence is that the question when an infinitely divisible process is a semimartingale can often be reduced to a path property, when a certain associated infinitely divisible process is of finite variation.  This gives the key to characterize the semimartingale property for many processes of interest.  Along these lines, using Basse-O'Connor and Rosi\'nski [2013, \emph{Stochastic Process.\  Appl.~123}(6)],  we characterize semimartingales within a large class of stationary increment infinitely divisible processes; this class includes many infinitely divisible processes of interest, including linear fractional processes, mixed moving averages, and supOU processes, as particular cases.
 The proof of the main theorem relies on series representations of jumps of c\`adl\`ag infinitely divisible processes given in Basse-O'Connor and Rosi\'nski [2013, \emph{Ann.\ Probab.~41}(6)] combined with techniques of stochastic analysis.



\end{abstract}

\section{Introduction}

A process $\mathbf{X}=(X_t)_{t\geq 0}$ on a filtered probability space $(\Omega, \mathcal{F}, \Fi=(\F_t)_{t\ge 0}, \mathbb{P})$  is called a semimartingale (relative to the filtration $\Fi$) if it admits a decomposition 
 \begin{equation}\label{decomp-X}
 X_t=X_0+M_t+A_t,\quad t\geq 0,
 \end{equation}
 where $\M= \p M$ is a c\`adl\`ag local martingale, $\A= \p A$ is a c\`adl\`ag adapted process of finite variation, $M_0=A_0=0$ and $X_0$ is $\F_0$-measurable. $\X$ is called a special semimartingale if \eqref{decomp-X} holds with $\A$ being also predictable.  In that case decomposition \eqref{decomp-X} is unique and is called the canonical decomposition of $\X$. We refer to \citet{Jacod_S} and \citet{Protter} for basic properties of semimartingales. 
 
  
  Semimartingales play a crucial role in stochastic analysis as they form the class of   \emph{good integrators}  for the It\^o  stochastic integral,  cf.\ the  Bichteler--Dellacherie Theorem \cite{Bichteler} and  \cite{B-D-direct-proof}. 
Semimartingales also play a fundamental role in mathematical finance. Roughly speaking, the (discounted) asset price process  must be a semimartingale in order to preclude arbitrage opportunities, see   
\citet[Theorems 1.4, 1.6]{B-D-direct-proof} for details, see also \cite{K-P}.
The question whether a given process is a semimartingale is also of importance in stochastic modeling, where long memory processes with possible jumps and high volatility are considered as driving processes for stochastic differential equations. Examples of such processes include various fractional, or more generally, Volterra processes driven by L\'evy processes.

The problem of identifying semimartingales within given classes of stochastic processes has a long history. For Markov processes this problem was  studied by \citet{cin} and \cite{mij, sch}, and in the context of Gaussian processes, it  was intensively studied in 1980s. 
\citet{Galchuk} investigated Gaussian semimartingales addressing a question posed by Prof.\ A.N.\ Shiryayev. Key results  on Gaussian semimartingales are due to \citet{Gau_Qua}, \citet{Stricker_gl},    \citet{Knight},  \citet{Yor:sem}, \citet[Ch.~4.9]{Lip_S} and \cite{Andreas3, Andreas2,  Andreas1, Basse_Graversen,   Cherny, Emery}.  
 Stricker's theorem \cite[Th\'eor\`eme~1]{Stricker_gl} is probably the most  fundamental result on Gaussian semimartingales and it is used to obtain all of the above cited  results (except \cite{Gau_Qua}, which it extends). 
An important question when certain Gaussian semimartingales admit an equivalent local martingale measure was studied by \citet{Patrick}.

Throughout this paper, if $\X$ is  a process with index set $T \subset \R$, then
$\Fi^X=(\F^X_t)_{t\geq 0}$ denotes its natural filtration; i.e., the least
 filtration satisfying the \emph{usual conditions} such that  $\sigma(X_s\!:s\leq t, s \in T)\subseteq \F^X_t$, $t\geq 0$.  
 
  \begin{theorem*}[Stricker's theorem]
  Let $\X$ be a symmetric Gaussian process. Then $\X$ is a semimartingale relative to its natural filtration $\Fi^X$ if and only if it admits a decomposition \eqref{decomp-X}, where $(\X, \M,\A)$ are jointly symmetric Gaussian, $\M$ has independent increments and $\A$  is  a predictable  process of finite variation. 
  In this case, $\X$ is a special semimartingale and  \eqref{decomp-X} is the canonical decomposition of $\X$. 
 \end{theorem*}

In this paper we investigate if (and how) Stricker's theorem can be generalized to the much larger class of infinitely divisible processes, which includes Gaussian, stable and other processes of interest (see Section~\ref{appl} for specific examples). Recall that a process $\mathbf{X}=(X_t)_{t\in T}$ is said to be infinitely divisible if  all its finite dimensional distributions are infinitely divisible, and it is  called symmetric if $\bf X$ and $-\bf X$  have the same finite dimensional distributions. 

We will now do a preliminary analysis of this problem to gain more intuitions. There are two key features of the decomposition \eqref{decomp-X} in the Gaussian case. The first one is that components $\M$ and $\A$ of the canonical decomposition are in the same distributional class as $\X$, both are Gaussian. The second one is that $\M$ is a process with independent increments. 
The  following two examples show that we cannot hope to get a direct extension of Stricker's theorem. The  first one shows that the processes $\M$ and $\A$ in the canonical decomposition \eqref{decomp-X} of an infinitely divisible semimartingale are not infinitely divisible in general.  

\begin{example}\label{ex-1}
Let $\X=(X_t)_{t\geq 0}$ be the symmetric  infinitely divisible process given by 
\begin{equation}
X_t=\begin{cases} 
U+V & 0\leq t<1, \\
V & \ t\geq 1,
 \end{cases}  
\end{equation}
where random variables $U$ and $V$ have standard Gaussian and standard Laplace distributions, respectively, and $U$ and $V$  are independent. Then $\X$ is a special semimartingale relative to the natural filtration $\Fi^X$,
but processes $\M$ and $\A$ in its canonical decomposition \eqref{decomp-X} are not infinitely divisible (see Appendix A for details). 
\end{example}

The second  example shows that  $\M$ in the canonical decomposition \eqref{decomp-X} of an infinitely divisible semimartingale need not have independent increments. 

\begin{example}\label{ex-2}
Let $\X=(X_t)_{t\geq 0}$ be the symmetric infinitely divisible process given by $X_t=\sum_{k=1}^N B_k(t)$, where $\{B_k(t):t\geq 0\}$ are independent standard Brownian motions and $N$ is a Poisson random variable independent of $\{B_k(t):t\geq 0,\,k\in \N\}$. Then $\X$ is a special semimartingale relative to $\Fi^X$, with the canonical decomposition \eqref{decomp-X} given by $M_t=X_t$ and  $A_t=0$. Process $\M$ does not have independent increments (see Appendix A for details). 
\end{example}

This leads to the question:  \emph{What are the special properties of Gaussian processes  that make Stricker's theorem valid?}

The key to address this question is provided by Hida's multiplicity  theorem \cite[Theorem~4.1]{Hida}. We give it here in a simplified version  which suffices for our purposes, see Remark \ref{rem-HC}.

\begin{theorem}[Hida's multiplicity theorem]\label{thm-HC}
Let $\X= (X_t)_{t\ge0}$ be a symmetric Gaussian process which is right-continuous in probability.  Then there exist independent symmetric right-continuous in $L^2$ Gaussian processes $\mathbf{B}_j=(B_j(t))_{t\in \R}$, $j \le N \le \infty$, each $\mathbf{B}_j$ having independent increments and $B_j(0)=0$, such that for each  $t\ge 0$  $\F^X_t = \vee_{j}  \F^{B_j}_t$  and
  \begin{equation}\label{H-C-rep-eq-36}
   X_t = \sum_{j=1}^N  \int_{-\infty}^t f_j(t, s) \, dB_j(s) \qquad \text{a.s.}  
   \end{equation}
Here $(f_j(t,\cdot))_{t\ge 0}$ is a family of deterministic functions such that for every $t\ge 0$ \\ $\int_{-\infty}^t f(t,s)^2 \, m_j(ds)< \infty$, where $m_j(ds)=\E [B_j(ds)^2]$. 
\end{theorem}

\begin{definition}\label{}
An infinitely divisible process $\mathbf{X}=(X_t)_{t\geq 0}$ is said to be  {\it representable} if there exist a countable generated measurable  space $V$, an infinitely divisible independently scattered  random measure $\Lambda$ on $\R\times V$,  and a family of measurable functions $\{\phi(t, \cdot)\}_{t\ge 0}$ on $\R\times V$  such that for every $t\ge 0$
\begin{equation}\label{rep-X-62}
X_t=\int_{(-\infty,t]\times V} \phi(t,u)\,\Lambda(du) \qquad \text{a.s.}
\end{equation}
The process $\X$ is said to be  {\it strictly representable} if \eqref{rep-X-62} holds for some $(\Lambda, \phi)$ as above and $\F^X_t=\F^{\Lambda}_t$ for every $t\geq 0$. Here   $\Fi^\Lambda=(\F^\Lambda_t)_{t\geq 0}$ denotes  the filtration generated by $\Lambda$; see Section~\ref{sec-def}  for the definition of $\Lambda$ and further pertinent definitions and related facts.
\end{definition}

As a corollary to Hida's multiplicity theorem it follows that Gaussian processes are strictly representable:

\begin{corollary}\label{col-HC} We have the following:
\begin{enumerate}[(i)]
\item \label{cor-1-yr}  Every symmetric right-continuous in probability Gaussian process $\X= (X_t)_{t\ge0}$ is strictly representable by some  symmetric Gaussian random measure $\Lambda$.
\item \label{cor-2-yr}  Every symmetric right-continuous in probability, or mean zero and right-continuous in $L^1$, infinitely divisible process $\X= (X_t)_{t\ge0}$ is  representable.
\end{enumerate}
 \end{corollary}
\begin{proof}
\eqref{cor-1-yr}: Applying Theorem \ref{thm-HC}, we may take in \eqref{rep-X-62}   $V=\{1,\dots,N\}$, when $N < \infty$ or $V=\N$ when $N=\infty$, a Gaussian random measure $\Lambda$ on $\R\times V$ determined by $\Lambda((a,b]\times \{j\})=B_j(b)-B_j(a)$, and $\phi(t, (s,j))=f_j(t,s)$. \eqref{cor-2-yr} follows by  
Proposition~\ref{representable}. 
\end{proof}

Typical infinitely divisible processes are defined by a stochastic integral as in \eqref{rep-X-62} with specific $\Lambda$ and $\phi$, so they are explicitly representable. Moreover, by Corollary~\ref{col-HC}\eqref{cor-2-yr}, every right-continuous in probability symmetric infinitely divisible process is representable. 
On the other hand, the strict representability may be difficult, if not impossible, to attain. For instance, processes given by 
Examples~\ref{ex-1} and \ref{ex-2} are representable but not strictly representable. The latter fact can easily be deduced from the next theorem but direct proofs are also possible, see the end of Example~\ref{ex-1} in Appendix~\ref{app}. 

The following result generalizes Stricker's theorem to infinitely divisible processes. It is a direct consequence of our main result, Theorem~\ref{thm1}.


\begin{theorem}\label{Hida}
Suppose that $\X=(X_t)_{t\geq 0}$ is a symmetric infinitely divisible process representable by a symmetric infinitely divisible random measure $\Lambda$. Then $\X$ is a semimartingale relative to the filtration $\Fi^\Lambda$ if and only if 
\begin{equation}\label{de-X-73}
X_t=X_0+M_t+A_t
\end{equation}
where $\M$ and $\A$ are infinitely divisible processes representable by $\Lambda$ such that  $\M$ is a c\`adl\`ag process with independent increments relative to $\Fi^{\Lambda}$,  $\A$ is  a  predictable c\`adl\`ag process of finite variation and $M_0=A_0=0$.  Decomposition \eqref{de-X-73} is unique in the class of processes representable by $\Lambda$. Furthermore, $\X$ is a special semimartingale if and only if \eqref{de-X-73} holds and $\M$ is a martingale with independent increments. 
\end{theorem} 

There is a slight difference between \eqref{de-X-73} and \eqref{decomp-X} in the meaning of  $\M$. In \eqref{de-X-73}, $\M$ is a process with independent increments which needs not be a (local) martingale. It could be further decomposed into a martingale and a process of finite variation leading to \eqref{decomp-X} but we would loose the predictability of $\A$ and uniqueness of the decomposition.  If $\X$ is a Gaussian semimartingale relative $\Fi^X$, then by Corollary~\ref{col-HC}\eqref{cor-1-yr} and Theorem~\ref{Hida}, $\X$ is a special semimartingale and $(\M, \A,\X)$  are jointly Gaussian, which gives Stricker's theorem. If $\X$ is a symmetric $\alpha$-stable process representable by a symmetric $\alpha$-stable random measure $\Lambda$, for example, then $\X$ is a semimartingale relative $\Fi^{\Lambda}$ if and only if it has a decomposition \eqref{de-X-73} into jointly symmetric $\alpha$-stable processes $\M$ and $\A$. 
If such $\X$ is strictly representable by a symmetric $\alpha$-stable random measure, then \eqref{de-X-73} gives the decomposition of $\X$ relative to its natural filtration.

Our proofs rely on different techniques than those used in  the  Gaussian case, see Remark~\ref{remark-tech}. We combine series representations of \ca\ infinitely divisible  processes with  detailed analysis of their jumps, which seems to be a new approach in this context. This technique is possible because such series representations converge uniformly a.s.\ on compacts, as shown in a recent work of \citet[Theorem 3.1]{Ito-Nisio-D}.

Section~\ref{sec-def} contains preliminary definitions and facts. Our main result, Theorem~\ref{thm1}, is stated and proved in Section~\ref{s-sem}, and the proof of Theorem~\ref{Hida} is given at the end of this section. Theorem~\ref{thm1} reduces the question when an infinitely divisible process is a semimartingale relative to $\Fi^{\Lambda}$  to the one when a certain associated infinitely divisible process is of finite variation.
In Section~\ref{appl} we use Theorem~\ref{thm1} to obtain explicit necessary and sufficient conditions for a large class of stationary increment infinitely divisible processes to be  semimartingales, see Theorems~\ref{thm-suf} and \ref{thm-nes} and their subsequent remarks. These results extend \citet[Theorem~6.5]{Knight} from Gaussian 
to infinitely divisible  processes, see Corollary~\ref{cor-levy}. We then apply Theorems~\ref{thm-suf} and \ref{thm-nes} to characterize the semimartingale property of various type of processes including linear fractional processes, moving averages, supOU processes and etc. These latter results generalize in a natural way results of \citet{Basse_Pedersen} and \citet{Be_Li_Sc}. 
Some supplementary material was moved to Appendices A and B.

\section{ Preliminaries}\label{sec-def}

In this section we will give more definitions and notation, and recall some facts that will be used throughout this paper. The material on infinitely divisible random measures and the related stochastic integral can be found in \citet{Rosinski_spec}. $(\Omega,\F,\P)$ will stand for a complete probability space, $(V, \mathcal V)$ will denote a countable generated measurable space, that is, the $\sigma$-algebra $\mathcal V$ is generated by countable many sets, and $\{V_n \}\subset \mathcal{V}$ will be a fixed sequence such that $V_n \uparrow V$. Define 
\begin{equation} \label{}
  \s = \big\{A \in \B(\R)\otimes \V: \ A \subset [-n, n] \times V_n \ \text{for some } n \ge 1 \big\}.
\end{equation}
Then $\s$ is a $\delta$-ring of subsets of $\R\times V$ such that $\sigma(\s)=\B(\R)\otimes \V$. For example, $\s$ can be the family of bounded Borel subsets of an Euclidean space. A stochastic process $\Lambda=\{\Lambda(A) \}_{A \in  \s}$ is said to be an (independently scattered) infinitely divisible random measure if 
\begin{enumerate}[(i)]
  \item for any sequence $(A_n)_{n\in \N} \subseteq \s$ of pairwise disjoint sets, $\Lambda(A_n)$, $n=1,2,\dots$ are independent and if $\bigcup_{n=1}^{\infty} A_n \in \s$, then $\Lambda(\bigcup_{n=1}^{\infty} A_n) = \sum_{n=1}^{\infty} \Lambda(A_n)$ a.s.;
  \item $\Lambda(A)$ has an infinitely divisible distribution for every $A \in \s$.
\end{enumerate}
$\Fi^{\Lambda}=(\mathcal{F}^{\Lambda}_t)_{t \ge 0}$ will denote the natural  filtration of $\Lambda$,  i.e., the least filtration satisfying the \emph{usual conditions} of right-continuity and completeness such that 
\begin{equation}\label{def_fil}
\sigma\Big(\Lambda(A): A\in \s,\, A\subseteq (-\infty,t]\times V\Big)\subseteq \F^\Lambda_t,\qquad t\geq 0.
 \end{equation} 

We will now recall deterministic characteristics of $\Lambda$ that will play a crucial role in this paper. From \cite[Proposition 2.4]{Rosinski_spec}, there exist measurable functions $\morf{b}{\R\times V}{\R}$ and $\morf{\sigma}{\R\times V}{\R_+}$, a $\sigma$-finite measure $\kappa$ on $\R\times V$, and a measurable family $\{\rho_{u}\}_{u\in \R\times V}$ of L\'evy measures on $\R$  such that for every $A \in \s$ and $\theta \in \R$
\begin{equation}\label{Lambda} 
 \log \E e^{i\theta\Lambda(A)}= \int_{A} \Big[ i\theta b(u)-\frac{1}{2}\theta^2 \sigma^2(u) +\int_{\R} \big(e^{i\theta x}-1-i\theta \lt x\rt\big) \, \rho_{u}(dx) \Big] 
 \,\kappa(d u)\, .
\end{equation}
Here $u = (s,v) \in \R\times V$ and 
\begin{equation}\label{tr}
 \lt x\rt= \frac{x}{|x| \vee 1} = \begin{cases}
    x    &  \text{if } |x|<1\, ,  \\ 
    \mathrm{sgn}(x)    & \text{otherwise}
\end{cases}  
\end{equation}
is a truncation function. Given $b$, $\sigma^2$, $\kappa$, and $\{\rho_{u}\}_{u\in \R\times V}$ as above, there is an independently scattered random measure $\Lambda$ satisfying \eqref{Lambda} by Kolmogorov's Extension Theorem. According to \cite[Theorem 2.7]{Rosinski_spec}, the stochastic integral $\int_{\R\times V} f(u) \, \Lambda(du)$ of a measurable deterministic function $\morf{f}{\R\times V}{\R}$  exists  	if and only if
\begin{enumerate}[(a)]
   \item \label{i1} $\int_{\R\times V} |B(f(u), u)| \, \kappa(du) < \infty$,
      \item \label{i2} $\int_{\R\times V} K(f(u), u) \, \kappa(du) < \infty$,
\end{enumerate} 
where
\begin{align} \label{B}
B(x, u) = {}& x b(u) + \int_{\R} \big( \lt xy\rt - x \lt y\rt\big) \, \rho_{u}(dy)\quad \text{and}
\\
\label{K}
K(x, u) = {}& x^2 \sigma^2(u) + \int_{\R}  \lt xy\rt^{2} \, \rho_{u}(dy), \qquad x \in \R, \ u \in \R\times V.
\end{align}
When (a)--(b) hold, then $\int_{\R\times V} f(u)\,\Lambda(du)$ is  an infinitely divisible random variable. Moreover, if $f=f(t, \cdot)$ depends on a parameter $t\in T$, then $\big(\int_{\R\times V} f(t,u)\,\Lambda(du)\big)_{t\in T}$ is an infinitely divisible process.

We will also use the following definitions and notation. For a \ca\ function $\morf{g}{\R_+}{\R}$, the jump size of $g$ at $t$ is defined as $\Delta g(t)=\lim_{s\uparrow t, s<t} (g(t)-g(s))$ when $t>0$ and $\Delta g(0)=0$. 
If $\morf{X}{\Omega}{[0,\infty]}$ is a measurable function, then $[X]=\{(\omega,X(\omega))\!: \omega\in \Omega,\, X(\omega)<\infty\}$ denotes the graph of $X$. (Notice that \citet{Jacod_S} write $\lt X\rt$ for the graph of $X$.)  A random set $A\subseteq \Omega\times \R_+$ is said to be evanescent if the set  $\{\omega\in \Omega\!: \exists\, t\in \R_+ \text{ such that } (\omega,t)\in A\}$ is a $\P$-null set. For two random subsets $A$ and $B$ of $\Omega\times \R_+$, we say that $A\subseteq B$ up to evanescent if  $B\setminus A$ is evanescent.  Two processes $\X=(X_t)_{t\geq 0}$ and $\Y=(Y_t)_{t\geq 0}$ are said to be indistinguishable if the set $\{(\omega, t)\!: X_t(\omega) \ne Y_t(\omega) \}$ is evanescent.  We will write $\X=\Y$ when $\X$ and $\Y$  are indistinguishable. 

\bigskip

 \section{Infinitely divisible semimartingales} \label{s-sem}


In this section $\X=(X_t)_{t\geq 0}$ stands for a \ca\  infinitely divisible process which is representable by some infinitely divisible random measure $\Lambda$, i.e., a process of the form
\begin{equation}\label{def-X}
 X_t=\int_{(-\infty,t]\times V} \phi(t, u)\,\Lambda(du),  
\end{equation}
where $\morf{\phi}{\R_+\times (\R \times V)}{\R}$ is a measurable deterministic function and $\Lambda$ is specified by \eqref{Lambda}--\eqref{tr}. We assume that for every $u=(s,v) \in   \R\times V$, \, $\phi(\cdot, u)$ is \ca, cf. Remark~\ref{cadlag}. Let $B$ be given by \eqref{B}. We further assume that 
\begin{equation} \label{drift}
\int_{(0,t] \times V} \big| B\big(\phi(s, s, v), (s,v)\big) \big| \, \kappa(ds, dv) < \infty \quad \text{for every $t>0$.}
\end{equation}

\np
The following is the  main result of this section. 

\begin{theorem}\label{thm1}
Under the above assumptions  $\mathbf{X}$ is a semimartingale relative to the filtration  
$\mathbb{F}^{\Lambda}=(\mathcal{F}^{\Lambda}_t)_{t \ge 0}$ if and only if
\begin{equation}\label{dec}
   X_t = X_0 + M_t + A_t, \quad t \ge 0,
\end{equation}
where $\M= \p M$ is a   semimartingale with  independent increments given by the stochastic integral
\begin{equation} \label{M}
M_t = \int_{(0,t] \times V} \phi(s,(s,v))\,\Lambda(ds,dv), \quad t\ge 0,
\end{equation}
and $\mathbf{A}= \p A$ is a predictable \ca\ process of finite variation  of the form 
\begin{equation} \label{A}
A_t = \int_{(-\infty,t] \times V} \big[\phi(t,(s,v)) -  \phi(s_{+}, (s,v))\big] \,\Lambda(ds,dv).
\end{equation}
Decomposition \eqref{dec}	 is unique in the following sense: If $\X=X_0+\M'+\A'$, where $\M'$ and $\A'$ are processes representable by $\Lambda$ such that  $\M'$ is a semimartingale with independent increments relative to $\mathbb{F}^{\Lambda}$ and $\A'$ is a predictable \ca\ process of finite variation, then $\M'=\M+g$ and $\A'=\A-g$ for some  c\`adl\`ag deterministic function $g$ of finite variation, where 
$\M$ and $\A$ are given by \eqref{M} and \eqref{A}. 

$\bf X$ is a special semimartingale if and only if \eqref{dec}--\eqref{A} hold and ${\E\abs{M_t}<\infty}$ for all $t>0$. In this case,  $(M_t-\E M_t)_{t\geq 0}$ is a  martingale and  
\begin{equation}
  X_t= X_0+ (M_t -\E M_t)+ (A_t+\E M_t),\quad t\geq 0
\end{equation}
  is the canonical decomposition of $\X$. 
\end{theorem}
\bigskip

In the next section we use Theorem~\ref{thm1}  to characterize the semimartingale property of various infinitely divisible processes with stationary increments. In the  following we conduct  the proofs of Theorems~\ref{thm1} and \ref{Hida}, but first  we  consider two remarks and an example.

\begin{remark}\label{cadlag}
{\rm 
If $\X$ given by \eqref{def-X} is a semimartingale relative $\Fi^\Lambda$, and $\Lambda$ satisfies the  non-deterministic condition
\begin{equation}\label{eq-kappa-42}
\kappa\big(u\in \R\times V\! : \sigma^2(u)=0,\, \rho_u(\R)=0
\big)=0, 
\end{equation}
then $\phi$ can be chosen such that $\phi(\cdot, u)$ is \ca\ for every $u=(s,v) \in   \R\times V$. The proof of this statement is given in the Appendix A.
}
\end{remark}

\begin{remark}\label{remark-sym}
{ \rm
Condition~\eqref{drift} is always satisfied when $\Lambda$ is symmetric. Indeed, in this case $B\equiv 0$. 
}
\end{remark}



\begin{example}
Consider the setting in Theorem~\ref{thm1} and suppose that  $\Lambda$ is an  $\alpha$-stable random measure and $\alpha\in (0,1)$.  Then $\bf X$ is a semimartingale with respect to $\Fi^\Lambda$ if and only if it is  of finite variation. This follows by Theorem~\ref{thm1} because the process $\bf M$ given by \eqref{M} is of finite variation. 
Indeed,  the L\'evy--It\^o decomposition of $\M$ (\cite[II, 2.34]{Jacod_S})  combined with  \cite[II, 1.28]{Jacod_S} show that $\M$ is of finite variation. 

\end{example}

\medskip

\begin{proof}[Proof of Theorem~\ref{thm1}]

The sufficiency is obvious. To show the necessary part we need to show that a semimartingale $\X$ has a decomposition \eqref{dec} where the processes $\bf M$ and $\bf A$ have the stated properties. We will start by considering the case where $\Lambda$ does not have a Gaussian component, i.e.\ $\sigma^2=0$. We may and will assume that $\phi(0,u)=0$ for all $u$ corresponding to $X_0=0$ a.s., and that $\phi(t,(s,v))=0$ for $s>t$ and $v\in V$. 

\medskip

\emph{Case 1. $\Lambda$ has no Gaussian component}:    We  divide the proof  into the following six steps. 

\medskip

\emph{Step~1}: 
Let $X^0_t = X_t - \beta(t)$, with 
\begin{equation}
 \beta(t) = \int_{U} B\big(\phi(t, u), u\big) \, \kappa(d u),\qquad U=\R\times V. 
\end{equation}
We will give the  series representation for $\X^0$ that will be crucial for our considerations. To this end, define for $s\neq 0$ and $u \in U=\R\times V$ 
\begin{equation}\label{}
R(s, u) = \begin{cases}
\inf\{ x>0: \rho_u(x,\infty) \le s\} \qquad\qquad   & \text{if } s>0, \\ 
\sup\{ x<0: \rho_u(-\infty, x) \le -s\}  & \text{if } s<0. 
\end{cases}
\end{equation}
Choose a probability measure $\tilde \kappa$ on $U$ equivalent to $\kappa$, and let $h(u)= \frac{1}{2}(d \tilde \kappa/d\kappa)(u)$.
By an extension of  our probability space if necessary, \citet{Rosinski_series_point}, Proposition~2 and Theorem~4.1, shows that there exists three independent sequences
$(\Gamma_i)_{i\in \N}$,  $(\epsilon_i)_{i\in \N}$,   and $(T_i)_{i\in \N}$, where $\Gamma_i$ are partial sums of i.i.d.\ standard exponential random variables, $\epsilon_i$ are i.i.d.\ symmetric Bernoulli random variables, and $T_i=(T_i^1, T_i^2)$ are i.i.d.\ random variables in $U$ with the common distribution $\tilde \kappa$, such that for every $A\in \s$, 
\begin{equation}\label{rep_lambda}
 \Lambda(A)= \nu_0(A)+  \sum_{j=1}^\infty   \big[R_j\1_A(T_j)-\nu_j(A)\big] \qquad \text{a.s.}
\end{equation}
where $R_j=R(\epsilon_j\Gamma_j h(T_j),T_j)$, \ $\nu_0(A)= \int_A b(u) \, \kappa(d u)$,  and for $j\ge 1$
\begin{equation}\label{}
\nu_j(A) = \int_{\Gamma_{j-1}}^{\Gamma_j} \E \lt R(\epsilon_1 r h(T_1),T_1)\rt \1_A(T_1) \, dr.
\end{equation}
It follows by the same argument that 
\begin{equation} \label{sY0}
X^0_t = \sum_{j=1}^{\infty} \big[ R_j \phi(t, T_j) - \alpha_j(t) \big] \qquad \text{a.s.},
\end{equation}
where 
\begin{equation} \label{}
\alpha_j(t) = \int_{\Gamma_{j-1}}^{\Gamma_j} \E \lt R(\epsilon_1 r h(T_1),T_1) \phi(t, T_1)\rt  \, dr.
\end{equation}

\medskip

\emph{Step~2}: Set $J=\{t\geq 0\!: \kappa(\{t\}\times V)>0\}$, 
\begin{equation}
T^{1,c}_i=T_i^1\1_{\{T_i^1\in \R_+\setminus J\}}\qquad \text{and}\qquad  T^{1,d}_i=T_i^1\1_{\{T_i^1\in J\}}.
\end{equation}
 Since $\kappa$ is a $\sigma$-finite measure the set  $J$ is countable. Furthermore,  $\P(T^{1,c}_i=x)=0$ for all $x>0$ and $T^{1,d}_i$ is discrete. 
We will show that  for every $i \in \N$
\begin{equation} \label{eq:T2}
\Delta X_{T_i^{1,c}} = R_i \phi(T_i^{1,c}, T_i) \quad \text{a.s.\ }
\end{equation}
 Since $\X$ is \ca,  the series 
 \begin{equation} \label{eq:Yep}
X_t^0=\sum_{j=1}^{\infty} \big[ R_j \phi(t, T_j) - \alpha_j(t) \big] 
\end{equation} 
converges uniformly for $t$ in compact intervals a.s., cf.\ \citet[Corollary~3.2]{Ito-Nisio-D}. 
Moreover, $\beta$ is \ca, see \cite[Lemma~3.5]{Ito-Nisio-D},  and by Lebesgue's dominated convergence theorem it follows that 
$\alpha_j$, for $j\in \N$, are \ca\ as well. Therefore, with probability one,
\begin{equation} \label{}
\Delta X_t = \Delta \beta(t) + \sum_{j=1}^{\infty} \big[ R_j \Delta \phi(t, T_j) - \Delta \alpha_j(t) \big] \quad \text{for all} \ t>0.
\end{equation}
Hence, for every $i \in \N$ almost surely
\begin{equation} \label{eq:T1}
\Delta X_{T_i^{1,c}} = \Delta \beta(T_i^{1,c}) + \sum_{j=1}^{\infty} \big[ R_j \Delta \phi(T_i^{1,c}, T_j) - \Delta \alpha_j(T_{i}^{1,c}) \big] 
\end{equation}
Since $\beta$ has at most  countable many   discontinuities (it is c\`adl\`ag), with probability one $T_i^{1,c}$ is a continuity point of $\beta$ since $\P(T^{1,c}_i=x)=0$ for all  $x>0$. Hence $\Delta \beta(T_i^{1,c})=0$ a.s. Since $(\Gamma_j)_{j\in \N}$ are independent of $T_i^{1,c}$, the argument used for $\beta$ also yields $\Delta \alpha_j(T_i^{1,c})=0$ a.s. By \eqref{eq:T1} this proves
\begin{equation}\label{eq:T1-a}
 \Delta X_{T_i^{1,c}} = 
\sum_{j=1}^{\infty} R_j \Delta \phi(T_{i}^{1,c}, T_j).
\end{equation}
Furthermore, for $i\ne j$ we get
\begin{align} \label{}
\P( \Delta \phi(T_i^{1,c}, T_j) \ne 0) &= \int_{U} \P( \Delta \phi(T_i^{1,c}, T_j) \ne 0\, |\, T_j= u) \, \tilde \kappa(d u) \\
&= \int_{U} \P( \Delta \phi(T_i^{1,c}, u) \ne 0 ) \, \tilde \kappa(d u) = 0
\end{align}
again because $\phi(\cdot, u)$ has only countably many jumps and the distribution of $T_i^{1,c}$ is continuous on $(0,\infty)$. 
If $j=i$ then
\begin{equation} \label{}
\Delta \phi(T_i^{1,c}, T_i) = \lim_{h \downarrow 0,\, h>0} \big[\phi(T_i^{1,c}, (T_i^1, T_i^2)) - \phi(T_i^{1,c}-h, (T_i^1, T_i^2)) \big] = \phi(T_i^{1,c}, T_i) 
\end{equation}
as $\phi(t,(s,v))=0$ whenever $t<s$ and $v\in V$. This simplifies \eqref{eq:T1-a} to \eqref{eq:T2}.

\medskip

\emph{Step~3}: Next we will show that 
$\M$, defined  in \eqref{M}, is a well-defined \ca\ process satisfying
\begin{equation}\label{eq-MY}
 \Delta M_{T^{1,c}_i} =\Delta X_{T^{1,c}_i}\quad \text{a.s.\ for all }i\in \N.
\end{equation} 
Since any semimartingale has finite quadratic variation, we have in  particular
\begin{align}\label{eq-fit-var72}
\sum_{0<s\le t} \big(\Delta X_s\big)^2<\infty\qquad \text{a.s.}
\end{align}
Let $\X'$ be an  independent copy of $\X$ and set $\tilde \X=\X-\X'$. Let  $\bar R_j= R(\epsilon_j \Gamma_j h(T_j)/2,T_j)$ and $(\xi_j)_{j\in \N}$ be  i.i.d.\ symmetric Bernoulli random variables 
defined on a probability space $(\Omega',\F',\P')$.  By   \citet[Theorem 2.4]{Rosinski_On_Series} it follows that for all $t\geq 0$ the series 
\begin{equation}
\bar X_t=\sum_{j=1}^\infty \xi_j
\bar R_j\phi(t,T_j)
\end{equation}
defined on $\Omega\times \Omega'$ converge a.s.\ under $\P\otimes \P'$ and  $\bar \X$ equals $\tilde \X$ in finite  dimensional  distributions. 
  Thus $\bar \X$ has a c\`adl\`ag modification satisfying 
   \begin{equation}\label{eq7242}
 \sum_{s\in (0,t]} \big(\Delta \bar X_s)^2<\infty\qquad \P\otimes\P'\text{-a.s.}
 \end{equation}
By \citet[Corollary~3.2]{Ito-Nisio-D}, we have $\P\otimes \P'$-a.s.\ for all $t\geq 0$ that 
 \begin{equation}\label{tr63}
 \Delta \bar X_t=\sum_{j=1}^\infty \xi_j \bar R_j\Delta \phi(t,T_j).
 \end{equation}
By \eqref{eq7242} and \eqref{tr63}  we have  for $\P$-a.a.\ $\omega\in \Omega$ that 
\begin{align}
 \sum_{s\in A} Y_s^2< \infty \qquad \P'\text{-a.s.,}\qquad 
 {}&\text{where}\qquad 
Y_s=\sum_{j=1}^\infty a(s,j) \xi_j,\\
a(s,j)=  \bar R_j(\omega)\Delta\phi(s,T_j(\omega))\qquad 
{}&\text{and}\qquad  A=\cup_{j\in \N}\{s\in (0,t]\!: \Delta \phi(s,T_j(\omega))\neq 0\}.   
\end{align}
For a fixed $\omega\in \Omega$ as above, $A$ is a countable deterministic  set and ${\bf Y}=(Y_s)_{s\in A}$ is a Bernoulli/Rademacher  random element in $\ell^2(A)$  defined on $(\Omega',\F',\P')$. By \citet[Theorem~4.8]{Talagrand}, $\E'[\|Y\|_{\ell^2(A)}^2]<\infty$ which implies that 
\begin{equation}\label{eq842}
\infty>\E'\Big[ \sum_{s\in A} Y_s^2\Big]=\sum_{s\in A} \E'[ Y_s^2]=\sum_{s\in A} \sum_{j=1}^\infty a(s,j)^2= \sum_{j=1}^\infty \Big(
\sum_{s\in A} a(s,j)^2\Big).
\end{equation}
Eq.~\eqref{eq842} implies that $\P$-a.s.
\begin{equation}
\infty>\sum_{i:\,T_i^1\in (0,t]}  |\bar R_i\Delta \phi(T^1_i,T_i)|^2=\sum_{i:\,T_i^1\in (0,t]}  |\bar R_i\phi(T^1_i,T_i)|^2.
\end{equation}
Put for $t,r \ge 0$ and $(\epsilon,s,v) \in \{-1,1\} \times \R \times V$
\begin{equation} \label{}
H(t; r, (\epsilon,s,v)) = R\big(\epsilon r h(s,v)/2, (s,v)\big) \phi(s,(s,v)) \1_{\{0< s \le t\}}.
\end{equation}
The above bound shows that for each $t\ge 0$
\begin{equation} \label{}
\sum_{i=1}^{\infty}  |H(t; \Gamma_i, (\epsilon_i,T_i^1,T_i^2)) |^2 < \infty \quad \text{a.s.}
\end{equation}
That implies, by  \citet[Theorem 4.1]{Rosinski_series_point}, that  the following limit is finite
\begin{align*}
\lim_{n\to \infty} \int_0^n \E \lt H(t; r, (\epsilon_1,T_1^1,T_1^2))^2 \rt \, dr = \int_0^{\infty} \E \lt H(t; r, (\epsilon_1,T_1^1,T_1^2))^2 \rt \, dr.
\end{align*}
Evaluating this limit we get
\begin{align*}
\infty &> \int_0^{\infty} \E \lt R(\epsilon_1 r h(T_1)/2, T_1) \phi(T_i^1, T_i) \1_{\{0<T_i^1 \le t\}}\rt^2 \, dr \\
& = \int_0^{\infty} \int_{\R\times V} \E\lt R(\epsilon_1 r h(s,v)/2, (s,v)) \phi(s, (s,v)) \1_{\{0< s \le t\}}\rt^2 \, \tilde\kappa(ds,dv) \,dr \\
& = 4 \int_0^{\infty} \int_{\R\times V} \E\lt R(\epsilon_1 z, (s,v)) \phi(s, (s,v)) \1_{\{0< s \le t\}}\rt^2 \, \kappa(ds,dv) \,dz \\
& = 2\int_{\R\times V} \int_{\R} \lt x \phi(s, (s,v)) \1_{\{0< s \le t\}}\rt^2 \, \rho_{(s,v)}(dx) \,\kappa(ds,dv) \\
& = 2\int_{(0,t] \times V} \int_{\R} \min\{ |x \phi(s, (s,v))|^2, 1\} \, \rho_{(s,v)}(dx)\, \kappa(ds,dv). 
\end{align*}
Finiteness of this integral in conjunction with \eqref{drift} yield the existence of the stochastic integral
\begin{equation} \label{}
M_t = \int_{(0,t] \times V} \phi(s, s,v) \, \Lambda(ds,dv) 
\end{equation}
by \eqref{i1} and \eqref{i2} on page~\pageref{i1}. 
The fact that $\M$ has independent increments is obvious since $\Lambda$ is independently scattered. Furthermore, $\M$ is c\`adl\`ag in probability by the continuity properties of  stochastic integrals, and by Lemma~\ref{lem-cad-mod}  it has a c\`adl\`ag modification  which will also be denoted by $\M$.
 Let $(\zeta_t)_{t\geq 0}$ be the shift component of $\M$.  By   \eqref{drift} and the fact that 
\begin{equation}\label{def-a}
 \zeta_t=\int_{(0,t]\times V} B\big(\phi(s,s,v),(s,v)\big)\,\kappa(ds,dv),\qquad t\geq 0,
\end{equation}
see \cite[Theorem~2.7]{Rosinski_spec}, we deduce that  $(\zeta_t)_{t\geq 0}$ is of finite variation.   Therefore the independent increments of $\M$ and \cite[II, 5.11]{Jacod_S} show that  $\M$ is a semimartingale. For $t\geq 0$ we can  write $M_t$ as a series using the series representation \eqref{rep_lambda} of $\Lambda$. It follows that
\begin{equation} \label{se-M}
M_t =\zeta_t+ \sum_{i=1}^{\infty} \big[ R_i \phi(T_i^1, T_i) \1_{\{0<T_i^1 \le t\}} - \gamma_j(t)\big]
\end{equation}
where 
\begin{equation}\label{}
 \gamma_j(t)=\int_{\Gamma_{j-1}}^{\Gamma_j} \E \lt R(\epsilon_1 r h(T_1),T_1) \phi(T^1_1,T_1)\1_{\{ 0<T^1_j\leq t\}})\rt  \, dr. 
\end{equation}
By arguments as above we have $\Delta M_{T^{1,c}_i}= R_i\phi(T^{1,c}_i,T_i)$ a.s.\ and hence by \eqref{eq:T2} we obtain \eqref{eq-MY}. 

\medskip

\emph{Step~4}:  In the following we will show the existence of  a sequence $(\tau_k)_{k\in \N}$ of totally inaccessible stopping times such that  all local martingales $\mathbf{Z}=(Z_t)_{t\geq 0}$  with respect to $\Fi^\Lambda$ are purely discontinuous and up to evanescent
\begin{equation}\label{con-stop}
\{\Delta {\bf Z}\neq 0\}\subseteq (\Omega\times J)\cup (\cup_{k\in \N}[\tau_k]),\qquad \cup_{k\in \N}[\tau_k] \subseteq \cup_{k\in \N}[T^{1,c}_k].
\end{equation}
Recall that $\{\Delta {\bf Z}\neq 0\}$ denotes the random set $\{(\omega, t)\in \Omega\times \R_+\!: Z_t(\omega)\neq 0\}$ and  $J$ is the countable subset of $\R_+$ defined in Step~2. 
Set $\V_0=\{A\in \V:A\subseteq V_k\text{ for some }k\in \N\}$ where $(V_k)_{k\in \N}$ is given in the Preliminaries.  
To show \eqref{con-stop} choose a sequence $(B_k)_{k\geq 1}\subseteq \V_0$ of disjoint sets which generates $\mathcal V$ and for all  $k\in \N$ let ${\bf U}^k=(U^k_t)_{t\geq 0}$ be given by 
\begin{equation}\label{}
  U^k_t=\Lambda((0,t]\times B_k).
\end{equation} 
For $k\in \N$,  ${\bf U}^k$ is a c\`adl\`ag in probability infinitely divisible process with independent increments and has therefore a \ca\ modification by Lemma~\ref{lem-cad-mod} (which will also be denoted ${\bf U}^k$).  Hence   ${\bf U}=\{(U_t^k)_{k\in\N}:t\in \R_+\}$ is a  \ca\ $\R^\N$-valued process with no Gaussian component. Let $E=\R^\N\setminus\{0\}$. Then $E$ is a Blackwell space and  $\mu$ defined  by 
   \begin{equation}\label{}
 \mu(A)=\sharp\big\{t\in\R_+\!:(t,\Delta U_t)\in A\big\},\qquad A\in \mathscr{B}(\R_+\times E)
\end{equation} 
 is an extended Poisson random measure on $\R_+\times E$, in the sense of \cite[II, 1.20]{Jacod_S}. Let $\nu$ be the  intensity measure of $\mu$. We have that  $\Fi^\Lambda$ is the least filtration for which $\mu$ is an optional random measure. Thus according to   \cite[III, 1.14(b) and the remark after III, 4.35]{Jacod_S},  $\mu$ has the martingale representation property, that is for all real-valued local martingales ${\bf Z}=(Z_t)_{t\geq 0}$ with respect to $\Fi^\Lambda$ there  exists  a predictable function  $\phi$ from $\Omega\times \R_+\times E$ into $\R$ such that 
 \begin{equation}\label{mar-rep-N}
  Z_t = \phi*(\mu-\nu)_t,\quad t\geq 0 
 \end{equation} 
 (in \eqref{mar-rep-N} the symbol $*$ denotes integration with respect to $\mu-\nu$ as in \cite[II, 1.]{Jacod_S}). 
 Note that $\{t\geq 0\!:\nu(\{t\}\times E)>0\}\subseteq J$. 
By definition, see \cite[II, 1.27(b)]{Jacod_S},   $\bf Z$ is a purely discontinuous local martingale and $\Delta Z_t(\omega)=\phi(\omega,t,\Delta U_t(\omega))\1_{\{\Delta U_t(\omega)\neq 0\}}$ for $(\omega,t)\in \Omega\times J^c$ up to evanescent, which shows that 
 \begin{equation}\label{}
 \{ \Delta {\bf Z}\neq 0\}\subseteq (\Omega\times J)\cup \{ \Delta {\bf U}\neq 0\}\qquad \text{up to evanescent.}
 \end{equation}
    Lemma~\ref{lem-fixed-dis-234} and a
%
diagonal argument show the existence of  a sequence of totally inaccessible stopping times $(\tau_k)_{k\in \N}$ such that up to evanescent
    \begin{equation}\label{ex_sw}
\{\Delta {\bf U}\neq 0\}= (\Omega\times J)\cup (\cup_{k\in \N} [\tau_k]).
\end{equation}  
Arguing as in Step~2  with $\phi(t,(s,v))=\1_{(0, t]}(s)\1_{ B_k}(v)$ shows that  with probability one
\begin{equation}\label{jumps-U-12}
  \Delta U^k_t
 = \Delta \zeta(t) + \sum_{j=1}^{\infty} \big[ R_j  \1_{\{t=T^1_j\}} \1_{\{T^2_j\in B_k\}}- \Delta \gamma_j(t) \big] \quad \text{for all} \ t>0
\end{equation} 
where 
\begin{align}\label{}
 \xi(t)={}&  \int_{\R\times V} \1_{\{0\leq s\leq t\}}\1_{\{v\in B_k\}}b(s,v)\, \kappa(ds,dv),\\ 
 \gamma_j(t)={}&  \int_{\Gamma_{j-1}}^{\Gamma_j} \E \lt R(\epsilon_1 r h(T_1),T_1) \1_{\{ T^1_j\leq t\}}\1_{\{T^2_j\in B_k\}})\rt  \, dr.
\end{align}
The functions $\xi$ and $\gamma_j$, for $j\in \N$,  are continuous on $\R_+\setminus J$  and  hence with probability one
\begin{equation}\label{eqdasfas}
 \Delta U^k_t= \sum_{j=1}^{\infty} R_j  \1_{\{t=T^1_j\}} \1_{\{T^2_j\in B_k\}} \qquad \text{for all }t\in \R_+\setminus J.
\end{equation}
Since each $\tau_k$ is totally inaccessible and $J$ is countable, we have $\P(\tau_k\in J)=0$. Hence by \eqref{eqdasfas} we conclude that 
\begin{equation}\label{sub-T}
 \cup_{k\in \N} [\tau_k]\subseteq \cup_{k\in \N} [T^{1,c}_k]\qquad \text{up to evanescent.} 
\end{equation}
This completes the proof of Step~4.

\medskip 

\emph{Step~5}:
Fix $r\in \N$ and let  $\X'=(X'_t)_{t\geq 0}$ be given by 
\begin{equation}\label{}
 X' _t=X_t-\sum_{s\in (0,t]} \Delta X_s\1_{\{\abs{\Delta X_s}>r\}}.
\end{equation}
We will show that $\X'$ is a special semimartingale with martingale component $\M'=(M'_t)_{t\geq 0}$ given by 
\begin{equation}\label{def-M'}
M_t'=\tilde M_t-\E \tilde M_t \quad \text{where}\quad \tilde M_t=M_t-\sum_{s\in (0,t]}\Delta M_s \1_{\{\abs{M_s}>r\}}.
\end{equation}
 Recall that $\bf M$ is given by \eqref{M}. By \cite[II, 5.10 c)]{Jacod_S} it follows that $\M'$ is a martingale (and well-defined). 
The process   $\X'$ is a special semimartingale since its jumps are bounded by $r$ in absolute value; denote by $\bf W$ and $\bf N$ the finite variation and martingale compnents, respectively,  in the canonical decomposition $\X'=X_0+\mathbf{W}+\mathbf{N}$ of $\X'$. That is, we want to show  that $\mathbf{N}= \M'$. By \eqref{eq-MY} we have   for all $i\in \N$ 
  \begin{equation}\label{eq-M'}
    \Delta M'_{T_i^{1,c}}=\Delta M_{T_i^{1,c}}\1_{\{\abs{\Delta M_{T_i^{1,c}}}\leq r\}}=\Delta X_{T_i^{1,c}}\1_{\{\abs{\Delta X_{T_i^{1,c}}}\leq r\}}=\Delta X'_{T_i^{1,c}} \qquad \text{a.s.}
  \end{equation}
Let $(\tau_k)_{k\in \N}$ be a sequence of 
 totally inaccessible stopping times satisfying \eqref{con-stop} for both ${\bf Z}={\bf N}$ and ${\bf Z}=\M'$. Since $\textbf W$ is predictable and $\tau_k$ is a totally inaccessible stopping time we have that $\Delta W_{\tau_k}=0$ a.s.\ cf.\ 
 \cite[I, 2.24]{Jacod_S} and hence 
   \begin{equation}\label{eq_N_j}
 \Delta N_{\tau_k}=\Delta X_{\tau_k}'-\Delta W_{\tau_k}=\Delta X_{\tau_k}'=\Delta M_{\tau_k}'\qquad \text{a.s.}
 \end{equation}
 the last equality follows by \eqref{eq-M'} and the second inclusion in \eqref{con-stop}. 
 
Since  $J$ is countable we may find a set $K\subseteq \N$ such that $J=\{t_k\}_{k\in K}$.
 Next we will show that for all $k\in K$
 \begin{equation}\label{eq-683}
 \Delta N_{t_k}=\Delta M_{t_k}\qquad \text{a.s.} 
 \end{equation}
  By linearity,   $\A$, define in \eqref{A},   is a well-defined \ca\ process. For all $k\in K$ we have almost surely
\begin{align}
A_{t_k}={}&  \int_{(-\infty,t_k]\times V} \big[\phi(t_k,(s,v))-\phi(s,(s,v))\big]\,\Lambda(ds,dv)
\\ = {}&  \int_{(-\infty,t_k)\times V} \big[\phi(t_k,(s,v))-\phi(s,(s,v))\big]\,\Lambda(ds,dv)
\end{align}
which shows that $A_{t_k}$ is $\F_{t_k-}^\Lambda$-measurable.  Define a process ${\bf Z}=(Z_t)_{t\geq 0}$ by 
\begin{equation}\label{def-pro-z73}
Z_t=\sum_{k\in K} \big(\Delta A_{t_k}-\Delta W_{t_k}\big)\1_{\{t=t_k\}}.
\end{equation}
 Since $\Delta A_{t_k}-\Delta W_{t_k}$ is $\F_{t_k-}^\Lambda$-measurable for all $k\in K$, \eqref{def-pro-z73} shows that  $\bf Z$ is a predictable process. Let  $^{p} \Y$ denote  the predictable projection of any measurable process $\Y$, see \cite[I, 2.28]{Jacod_S}. Since $\bf Z$ is predictable   
\begin{equation}\label{eq-5823}
{\bf Z}=\,\! ^p {\bf Z}=\,\! ^p\big(\1_{\Omega\times J}(\Delta \A+\Delta {\bf W})\big) = \1_{\Omega\times J}  \, ^p (\Delta \A-\Delta {\bf W})=\1_{\Omega\times J} \, ^p (\Delta \M'-\Delta {\bf N})=0
\end{equation}
where  the third equality follows by \cite[I, 2.28(c)]{Jacod_S} and the fact that  $\Omega\times J$ is a predictable set,  the last equality follows by \cite[I, 2.31]{Jacod_S} and  the fact that $\M'$ and $\bf N$ are local martingales.  
Eq.~\eqref{eq-5823} shows that $\Delta A_t=\Delta W_t$ for all $t\in J$, which implies \eqref{eq-683}. 
 
 By \eqref{eq_N_j}, \eqref{eq-683} and the fact that 
 \begin{equation}\label{jumps-2324}
  \{\Delta {\bf N}\neq 0\}\subseteq (\Omega\times J)\cup (\cup_{k\in \N} [\tau_k]),\quad  \{\Delta \M'\neq 0\} \subseteq (\Omega\times J)\cup (\cup_{k\in \N} [\tau_k]) 
\end{equation}
we have shown that $\Delta {\bf N}=\Delta \M'$.
 By Step~4,  $\bf N$ and $\M'$ are purely discontinuous local martingale which implies that  ${\bf N}=\M'$, cf.\ \cite[I, 4.19]{Jacod_S}. 
This completes Step~5. 

\medskip

\emph{Step~6}: We will show that   $\A$  is a predictable \ca\ process of finite variation.   According to Step~5 the process  $\textbf{W}:=\X'-X_0-\M'$ is predictable and has \ca\ paths of finite variation.
Thus with $\mathbf{V}=(V_t)_{t\geq 0}$  given by 
  \begin{equation}
V_t= \sum_{s\in (0,t]} \Delta X_s\1_{\{\abs{X_s}>r\}}-\sum_{s\in (0,t]} \Delta M_s\1_{\{\abs{M_s}>r\}}
  \end{equation}
we have by the definitions of $\bf W$ and $\bf V$   that 
\begin{equation}\label{decomp-A-67}
  A_t= X_t-X_0-M_t=W_t+V_t-\E \tilde M_t.
  \end{equation}
  %
This shows that $\A$ has   \ca\ sample paths of finite variation. Next we will show that $\A$ is predictable. Since the processes $\bf W$, $\bf V$ and $\tilde \M$ depend on the truncation level $r$ they will be denoted  $\mathbf{W}^r$, $\mathbf{V}^{r}$ and $\tilde \M^r$ in the following.  As $r\to \infty$, $V_t^r(\omega)\to 0$ point wise in $(\omega,t)$, which by \eqref{decomp-A-67} shows that $W^r_t(\omega)-\E \tilde M^r_t \to A_t(\omega)$  point wise in $(\omega,t)$ as $r\to \infty$. For all $r\in \N$,    $(W^r_t-\E \tilde M^r_t)_{t\geq 0}$ is a predictable process, which implies  that  $\bf A$ is a point wise limit of predictable processes and hence predictable. This completes the proof of Step~6 and the proof of the decomposition \eqref{dec} in Case~1.
\medskip

%
\emph{Case 2. $\Lambda$ is symmetric Gaussian}:
Suppose  that $\Lambda$ is a symmetric  Gaussian random measure. By \citet[Theorem~4.6]{Andreas2} used on the sets  $C_t=(-\infty,t]\times V$,  $\X$ is a special semimartingale in $\Fi^{\Lambda}$ with martingale component $\M=(M_t)_{t\geq 0}$ given by 
   \begin{equation}\label{}
 M_t=\int_{(0,t]\times V} \phi(s,(s,v))\,\Lambda(ds,dv),\qquad t\geq 0,
\end{equation}
see \cite[Equation~(4.11)]{Andreas2}, which completes the proof in the   Gaussian case. 

\medskip

\emph{Case 3. $\Lambda$ is general}:
Let us observe that it is enough to show the theorem in the above two cases. We may decompose $\Lambda$ as $\Lambda = \Lambda_G + \Lambda_P$, where $\Lambda_G, \Lambda_P$ are independent, independently scattered random measures. $\Lambda_G$ is a symmetric Gaussian random measure characterized by \eqref{Lambda} with $b \equiv 0$ and $\kappa \equiv 0$ while $\Lambda_P$ is given by \eqref{Lambda} with $\sigma^2 \equiv 0$. Observe that 
\begin{equation} \label{eq-fil}
\Fi^{\Lambda} = \Fi^{\Lambda_G} \vee \Fi^{\Lambda_P},
\end{equation}
which can be deduced  from the L\'evy-It\^o decomposition, see \cite[II, 2.35]{Jacod_S},  used on the processes $\Y=(Y_t)_{t\geq 0}$ of the form $Y_t=\Lambda((0,t]\times B)$ where  $B\in \V_0$ ($\V_0$ is defined on page~\pageref{con-stop}).  We  have $\X = \X^G + \X^P$, where $\X^G$ and $\X^P$ are defined by \eqref{def-X} with $\Lambda_G$ and $\Lambda_P$ in the place of $\Lambda$, respectively. Since $(\Lambda, \X)$ and $(\Lambda_P -\Lambda_G, \X^P- \X^G)$ have the same distributions,
the process $\X^P -\X^G$ has a modification which is a semimartingale with respect to $\Fi^{\Lambda_P-\Lambda_G}= \Fi^{\Lambda_P} \vee \Fi^{-\Lambda_G}= \Fi^{\Lambda}$. 
Consequently, processes $\X^G$ and $\X^P$ have modifications which are semimartingales with respect to $\Fi^{\Lambda}$, and so, they are semimartingales relative to $\Fi^{\Lambda_G}$ and $\Fi^{\Lambda_P}$, respectively, and the general result follows from the above two  cases. 

\medskip

\emph{The uniqueness}: Let $\M, \M', \A$ and $\A'$ be as in the theorem. We will first show that $(\M, \M')$ is a bivariate process with independent increments relative to $\Fi^{\Lambda}$. To this aim, choose $0 \le s < t$ and $A_1, \dots, A_n \in \s$ such that $A_i \subset (- \infty, s] \times V$, $i\le n$, $n \ge 1$. Consider random vectors $\xi=(\xi_1,\xi_2):= (M_t -M_s, M_t' - M_s')$ and $\eta = (\eta_1,\dots,\eta_n):=(\Lambda(A_1),\dots, \Lambda(A_n))$. Since $\M$ and $\M'$ are processes representable by $\Lambda$, $(\xi, \eta)$ has an infinitely divisible distribution in $\R^{n+2}$. Since $\M$ and $\M'$ have independent increments relative to $\Fi^{\Lambda}$, $\xi_i$ is independent of $\eta_j$ for every $i\le 2$, $j\le n$. It follows from the form characteristic function and the uniqueness of L\'evy-Khintchine triplets that the pairwise independence between blocks of jointly infinitely divisible random variables is equivalent to the independence of blocks (this is a straightforward extension of \cite[Theorem~4]{Hudson-Trucker}). Therefore, $\xi$ is independent of $\eta$. We infer that $\xi$ is independent of $\F^{\Lambda}_s$, so that $(\M, \M')$ is a process with independent increments relative to $\Fi^{\Lambda}$,
 so is $\overline \M := \M'- \M$.

Since $\X=X_0+\M+\A=X_0+\M'+\A'$ by assumption, we have
\begin{equation} \label{}
 \overline\M = \M' - \M = \A'- \A,
\end{equation}
so that the independent increment semimartingale $\overline\M$ is predictable. For each $n\in \N$ define the truncated process  ${\overline \M}^{(n)} = (\overline M_t^{(n)})_{t\geq 0}$  by 
\begin{equation}
\overline M^{(n)}_t={\overline M}_t-\sum_{s\leq t} \Delta {\overline M}_s\1_{\{|\Delta {\overline M}_s|> n\}}.
\end{equation}
According to \cite[II, 5.10]{Jacod_S}, there exists a c\`adl\`ag deterministic function
${\bf g}_{n}$ of finite variation with $g_{n}(0)=0$ such that ${\overline \M}^{(n)}-{\bf g}_{n}$ is a martingale. Since  ${\overline \M}^{(n)}-{\bf g}_{n}$ is also predictable and of finite variation, ${\overline \M}^{(n)}={\bf g}_{n}$, cf.\ \cite[I, 3.16]{Jacod_S}. Letting $n\to \infty$ we obtain that $\overline\M$ is deterministic and obviously c\`adl\`ag and of finite variation.

\medskip

\emph{The special semimartingale part}:
To prove the  part concerning the special semimartingale property of $\X$ we note that the process $\bf A$ in \eqref{A} is a special semimartingale since it is a predictable c\`adl\`ag process  of finite variation. Thus $\bf X$ is a special semimartingale if and only if $\bf M$ is special semimartingale. Due to the independent increments,  $\bf M$ is a special semimartingale if and only if $\E\abs{M_t}<\infty$ for all $t>0$, cf.\  \cite[II, 2.29(a)]{Jacod_S}, and in that case $M_t=(M_t-\E M_t)+\E M_t$ is the  canonical decomposition of $\bf M$. This completes the proof.
\end{proof}

\bigskip

\begin{proof}[Proof of Theorem~\ref{Hida}]
We only need to prove the  \emph{only if}-implication. Suppose that $\X$ is a semimartingale with respect to $\Fi^\Lambda$. 
According to Remark~\ref{cadlag} we may and do choose  $\phi$ such that $t\mapsto \phi(t,u)$ is c\`adl\`ag of all $u$.  By Remark~\ref{remark-sym}, assumption \eqref{drift} is satisfied and hence by letting $\M$ and $\A$ be defined by \eqref{M} and \eqref{A}, respectively, we obtain the representation of $\X$ as claimed in Theorem~\ref{Hida}. To show the uniqueness part we note that by symmetric, the deterministic function $g$ in Theorem~\ref{thm1}  satisfies that $g(t)$ equals $-g(t)$ in law, which implies that $g(t)=0$ for all $t\geq 0$. Since the expectation of a symmetric random variable is zero whenever it exists the last part regarding the special semimartingale property follows as well. 
\end{proof}

\medskip

\begin{remark}\label{remark-tech}
{ \rm
We conclude this section by recalling that the proof of any of the results on Gaussian semimartingales $\X$ mentioned in the Introduction relies on the approximations of the finite variation component $\A$ by discrete time Doob--Meyer decompositions $\A^n=(A^n_t)_{t\geq 0}$ given by 
\begin{equation}
A^n_t=\sum_{i=1}^{[2^n t]} \E[X_{i2^{-n}}-X_{(i-1)2^{-n}}|\F_{(i-1)2^{-n}}],
\qquad t\geq 0
\end{equation}
and showing that the convergence $\lim_n A^n_t=A_t$ holds in an appropriate sense, see \cite{Stricker_gl, Meyer_app}.
This technique does not  seem  effective in the non-Gaussian situation since it relies on  strong integrability properties of  functionals of $\X$, which are in general not present and can not be obtained by stopping arguments.   
}
\end{remark}

\section{Some stationary increment semimartingales} \label{appl}


In this section we consider infinitely divisible  processes which are stationary increment mixed moving averages (SIMMA).
Specifically, a process $\X=\p X$ is called a SIMMA process if it can be written in the form
\begin{equation} \label{eq:mma}
X_t=\int_{\R \times V} \big[f(t-s, v)- f_0(-s, v)\big]\, \Lambda(ds, dv),\qquad t\geq 0,
\end{equation} 
where the functions $f$ and $f_0$ are deterministic measurable such that  $f(s, v) = f_0(s,v) = 0$ whenever $s<0$, and $f(\cdot,v)$ is \ca\ for all $v$.  $\Lambda$ is an independently scattered infinitely divisible random measure that is invariant under translations over $\R$. If $V$ is a one-point space (or simply, there is no $v$-component in \eqref{eq:mma}) and $f_0=0$, then \eqref{eq:mma} defines a moving average (a mixed moving average for a general $V$, cf. \cite{mixed_ma_R}). If $V$ is a one-point space  
and $f_0(x)=f(x)=x_+^\alpha$ for some $\alpha\in \R$,  then $\X$ is  a fractional L\'evy process.

The finite variation property of SIMMA processes was investigated in \citet{fv-mma} and these results, together with Theorem \ref{thm1}, are crucial in our description of SIMMA semimartingales.

The random measure $\Lambda$ in \eqref{eq:mma} is as in \eqref{Lambda} but the functions $b$ and $\sigma^2$ do not depend on $s$ and the measure $\kappa$ is a product measure: $\kappa(ds,dv)=ds\,m(dv)$ for some $\sigma$-finite measure $m$ on $V$.  In this case, for $A\in \s$ and $\theta\in \R$  
\begin{align}\label{eq:M}
&\log \mathbb{E} e^{i\theta \Lambda(A)} \\
& \quad = \int_A \Big( i\theta b(v)-\frac{1}{2}\theta^2 \sigma^2(v) +\int_{\R} (e^{i\theta x}-1-iu\lt x\rt) \, \rho_v(dx)\Big)\, ds\,m(dv).
\nonumber
\end{align}
The function $B$ in \eqref{B} is independent of $s$, so that with $B(x,v)=B(x,(s,v))$ we have
\begin{equation}\label{B1}
   B(x, v) = xb(v) +  \int_{\R} \big( \lt xy\rt - x \lt y\rt\big) \, \rho_{v}(dy), \quad x \in \R, \ v \in V.
\end{equation}

The SIMMA process \eqref{eq:mma} is a special case of \eqref{def-X} if we take $\phi(t,(s,v))=f(t-s, v)- f_0(-s, v)$. Therefore, from Theorem \ref{thm1} we obtain:

\begin{theorem}\label{decomp-special}
 Suppose that  
\begin{equation} \label{drift-4}
\int_{V} \big| B\big(f(0, v), v\big) \big| \, m(dv) < \infty.
\end{equation}
Then  $\mathbf{X}$ is a semimartingale with respect to the filtration 
$\mathbb{F}^{\Lambda}=(\mathcal{F}^{\Lambda}_t)_{t \ge 0}$ if and only if
\begin{equation}\label{dec-levy}
   X_t = X_0 + M_t + A_t, \quad t \ge 0,
\end{equation}
where $\M= \p M$ is a L\'evy process given by  
\begin{equation} \label{M-levy}
M_t = \int_{(0,t] \times V} f(0,v)\,\Lambda(ds,dv), \quad t\ge 0, 
\end{equation}
and $\mathbf{A}= \p A$ is a predictable process of finite variation  given by 
\begin{equation} \label{A-levy}
A_t = \int_{\R \times V} [g(t-s,v) -  g(-s,v)] \,\Lambda(ds,dv)
\end{equation}
where $g(s, v)= f(s, v) - f(0, v) \1_{\{s\geq 0\}}$. 
%
\end{theorem}

Now we will give specific and closely related necessary and sufficient 
conditions on $f$ and $\Lambda$ that make $\X$ a semimartingale.

 \begin{theorem}[Sufficiency]\label{thm-suf} 
Let $\mathbf{X}=\p X$ be specified by \eqref{eq:mma}--\eqref{eq:M}. Suppose that 
\eqref{drift-4} is satisfied 
and that for $m$-a.e.\ $v \in V$,  $f(\cdot,v)$ is absolutely  continuous on $[0,\infty)$ with a derivative 
$\dot  f(s,v)=\frac{\partial }{\partial s}f(s,v)$ satisfying
  \begin{align}
& \int_V \int_{0}^\infty\big(\abs{\dot f(s,v)}^2\sigma^2(v)\big)\,ds
 \,m(dv)<\infty,   \label{int_1} \\
 \label{Cf}
  &  \int_V \int_0^\infty \int_\R \big(\abs{x\dot f(s,v)}\wedge
   \abs{x\dot f(s,v)}^2\big)\,\rho_v(dx)\,ds\,m(dv)<\infty.
\end{align}
Then $\bf X$ is a semimartingale with respect to $\Fi^\Lambda$.  
\end{theorem}

 \begin{proof}
 We need to verify the conditions of Theorem \ref{decomp-special}.
 With $g(s, v)= f(s, v) - f(0, v) \1_{\{s\geq 0\}}$ we have for $m$-a.e.\ $v \in V$, $g(\cdot,v)$ is absolutely continuous on $\R$ with derivative $\dot g(s,v)=\dot f(s,v)$ for $s > 0$ and $\dot g(s,v)=0$ for $s<0$. By  Jensen's inequality,  for each fixed $t>0$, the function
 \begin{equation}
(s,v)\mapsto g(t-s,v)-g(-s,v)=\int_0^t \dot g(u-s,v)\,du,
 \end{equation}
when substituted for $\dot  f(s,v)$ in \eqref{int_1}--\eqref{Cf}, satisfies these conditions. Indeed,  it is  straightforward to verify \eqref{int_1}. To verify  \eqref{Cf} we use the fact that $\psi\! :u\mapsto 2\int_0^{\abs{u}} (v\wedge 1)\,dv$ is convex and satisfies $\psi(u)\leq \abs{ux}\wedge \abs{ux}^2\leq 2\psi(u)$. In particular,  $(s,v)\mapsto g(t-s,v)-g(-s,v)$ satisfies  \eqref{i2} of the Introduction, and so does the function
 \begin{equation}
 (s,v)\mapsto f(0,v) \1_{(0, t]}(s)=g(t-s,v) -  g(-s,v) - [f(t-s,v) -  f(-s,v)].  
 \end{equation}
This fact together with assumption  \eqref{drift-4} guarantee that $\M$ of Theorem~\ref{decomp-special} is well-defined. Then $\A$ is well-defined by \eqref{dec-levy}. The process $\A$ is of finite variation by \cite[Theorem~3.1]{fv-mma} because $g(\cdot,v)$ is absolutely continuous on $\R$ and $\dot g(\cdot, v)= \dot f(\cdot, v)$ satisfies \eqref{int_1}--\eqref{Cf}.
  \end{proof}
%

 \begin{theorem}[Necessity]\label{thm-nes}
 Suppose   that $\X$ is a semimartingale with respect to $\Fi^\Lambda$ and 
 for $m$-almost every $v\in V$ we have either 
   \begin{equation}\label{invar-con}
\int_{-1}^1 \abs{x}\,\rho_{v}(dx)=\infty\quad \text{or}\quad 
  \sigma^2(v)>0.
 \end{equation}
 Then for $m$-a.e.\ $v$,     $ f(\cdot,v)$ is absolutely continuous on
 $[0,\infty)$ with a derivative $\dot f(\cdot,v)$ satisfying
 \eqref{int_1} and  
  \begin{align}
  \label{trunc_case} & \int_0^\infty\int_\R \big(\abs{x \dot
    f(s,v)}\wedge \abs{x \dot f(s,v)}^2\big)(1 \wedge
  x^{-2})\,\rho_v(dx)\, ds<\infty.   
  \end{align} 
  If, additionally,
\begin{equation}\label{eq:u0}
\limsup_{u \to \infty} \, \frac{u\int_{\abs{x}>u}
  \abs{x}\,\rho_v(dx)}{\int_{|x|\le u} x^2 \, \rho_v(dx)} < \infty
\qquad m\text{-a.e.} 
\end{equation}
then for $m$-a.e.\ $v$,
\begin{equation}\label{fdot_int}
 \int_{0}^\infty \int_{\R} ( |x{\dot f}(s,v)|^2 \wedge |x{\dot
   f}(s,v)|) \, \rho_v(dx)\, ds < \infty. 
 \end{equation}
Finally, if 
 \begin{equation}\label{eq:u00}
\sup_{v\in V}\sup_{u > 0} \, \frac{u\int_{\abs{x}>u}
  \abs{x}\,\rho_v(dx)}{\int_{|x|\le u} x^2 \, \rho_v(dx)} <\infty 
\end{equation} 
then  $\dot f$ satisfies \eqref{int_1}--\eqref{Cf}. 
  \end{theorem}
  
   \begin{proof}
 Assume that $\X$ is a semimartingale with respect to $\Fi^{\Lambda}$. 
 By a symmetrization argument we may assume that $\Lambda$ is a symmetric random measure. Indeed,  let $\Lambda'$ be an independent copy of $\Lambda$ and $\X'$ be defined by \eqref{eq:mma} with $\Lambda$ replaced by $\Lambda'$. Then $\X'$ is a semimartingale with respect to  $\Fi^{\Lambda'}$. By the independence, both $\X$ and $\X'$ are semimartingales with respect to  $\Fi^\Lambda\vee \Fi^{\Lambda'}$ and since $\Fi^{\Lambda-\Lambda'}\subseteq \Fi^\Lambda\vee \Fi^{\Lambda'}$, the process $\X-\X'$ is a semimartingale with respect to  $\Fi^{\Lambda-\Lambda'}$. This shows that we may assume that $\Lambda$ is symmetric. Then  \eqref{drift-4} holds since $B=0$.   
 
 By Theorem~\ref{decomp-special}  process $\bf A$ in \eqref{A-levy} is of finite variation. 
It follows from \cite[Theorem~3.3]{fv-mma}  that for $m$-a.e.\ $v$, $g(\cdot,v)$ is absolutely continuous on $\R$ with a derivative $\dot g(\cdot,v)$ satisfying \eqref{int_1} and \eqref{trunc_case}.  Furthermore $\dot g$ satisfies \eqref{fdot_int} under  assumption \eqref{eq:u0},  and under assumption \eqref{eq:u00}, $\dot g$ satisfies   \eqref{Cf}. Since $f(s,v)=g(s,v)+f(0,v)\1_{\{s\geq 0\}}$, $f(\cdot,v)$ is absolutely continuous on $[0,\infty)$ with a derivative $\dot f(\cdot,v)=\dot g(\cdot,v)$ for $m$-a.e.\ $v$  satisfying the conditions of the theorem.  
 \end{proof}

  \begin{remark}\label{cor-iff}
Theorem~\ref{thm-nes} becomes an exact converse to Theorem~\ref{thm-suf} when \eqref{invar-con} holds and either  \eqref{eq:u0} holds and $V$ is a finite set, or \eqref{eq:u00} holds. 
\end{remark}

  \begin{remark}\label{remark_1}
Condition~\eqref{invar-con} is in general necessary to deduce that $f$
has absolutely continuous sections. Indeed, let $V$ be a one point
space so that $\Lambda$ is generated by increments of a L\'evy process
denoted again by $\Lambda$. If \eqref{invar-con} is not satisfied,
then taking  $f=\1_{[0,1]}$ we get that $X_t=\Lambda_t-\Lambda_{t-1}$
is of finite variation and hence a semimartingale, but $f$ is not
continuous on $[0,\infty)$. 
\end{remark}

\smallskip

Next we will consider several consequences of Theorems~\ref{thm-suf} and \ref{thm-nes}. When there is no $v$-component, \eqref{drift-4} is always satisfied and $\Lambda$ is generated by a two-sided  L\'evy process.
In what follows, $\mathbf{Z}=(Z_t)_{t\in \R}$ will denote a non-deterministic two-sided  L\'evy process, with characteristic triplet $(b,\sigma^2,\rho)$,  $Z_0=0$ and natural filtration $\Fi^Z$.
\smallskip	

The following proposition characterizes fractional L\'evy processes which are semimartingales, and  completes results of \cite[Corollary~5.4]{Basse_Pedersen} 
and parts of \cite[Theorem~1]{Be_Li_Sc}.

\begin{proposition}[Fractional L\'evy processes]\label{fLp}
 Let $\gamma >0$,   $x_+:=\max\{x,0\}$ for $x\in \R$, $\mathbf{Z}$ be a L\'evy process as above, and   $\bf X$ be a fractional L\'evy process defined by  \begin{equation}\label{def-frac}
X_t=\int_{-\infty}^t \big\{(t-s)^\gamma_+-(-s)_+^\gamma\,\big\}\,dZ_s
\end{equation}
where   the stochastic integrals exist. 
Then  $\bf X$ is a semimartingale with respect to $\Fi^Z$ if and only if  $\sigma^2=0$, $\gamma\in (0,\tfrac{1}{2})$ and 
\begin{equation}\label{int-finite-786}
\int_\R \abs{x}^{\frac{1}{1-\gamma}}\,\rho(dx)<\infty. 
\end{equation}
\end{proposition}

\begin{proof}
First we notice that, as a consequence  of  $\X$ being   well-defined, $\gamma<\tfrac{1}{2}$ and 
\begin{equation}\label{equ-34}
\int_{\abs{x}>1} \abs{x}^{\frac{1}{1-\gamma}}\,\rho(dx)<\infty.
\end{equation}
Indeed, since the stochastic integral \eqref{def-frac} is well-defined,  \cite[Theorem~2.7]{Rosinski_spec} shows that 
\begin{equation}\label{Ros-Raj-23}
\int_{-\infty}^t \int_\R \big(1\wedge \abs{\{(t-s)^\gamma-(-s)^\gamma_+\}x}^2\big)\,\rho(dx)\,ds<\infty, \qquad t\geq 0.
\end{equation}
This implies that  $\gamma<\tfrac{1}{2}$ if $\rho(\R)>0$. A similar argument shows that $\gamma<\tfrac{1}{2}$ if $\sigma^2>0$, and thus, by the non-deterministic assumption on $\bf Z$,  we have shown  that $\gamma<\frac{1}{2}$.  Putting $t=1$ in~\eqref{Ros-Raj-23} and using  the estimate $\abs{(1-s)^\gamma-(-s)^\gamma_+}\geq \abs{\gamma (1-s)^{\gamma-1}}$ for $s\in (-\infty,0]$  we get
\begin{align}
\infty>{}&\int_{-\infty}^0 \int_\R \big(1\wedge \abs{\gamma(1-s)^{\gamma-1}x}^2\big)\,\rho(dx)\,ds\\
= {}& \int_\R \int_1^\infty \big(1\wedge \abs{\gamma s^{\gamma-1}x}^2\big)\,ds\,\rho(dx)
\\ \geq {}& \int_\R \int_{1\leq s\leq \abs{\gamma x}^{\frac{1}{1-\gamma}} }\,ds\,\rho(dx)\geq  \int_{\abs{\gamma x}>1} \big(\abs{\gamma x}^{\frac{1}{1-\gamma}} -1\big)\,\rho(dx),
\end{align}
which shows \eqref{equ-34}.

Suppose that $\bf X$ is a semimartingale. 
If  $\sigma^2>0$, then according to   Theorem~\ref{thm-nes},   $f$ is absolutely continuous on $[0,\infty)$ with a derivative $\dot f$ satisfying
\begin{align}
\int_0^\infty \abs{\dot f(t)}^2\,dt=\int_0^\infty \gamma^2 t^{2(\gamma-1)}\,dt<\infty\end{align}
which  is a contradiction and shows that $\sigma^2=0$. 
By the  non-deterministic assumption on  $\bf Z$ we have $\rho(\R)>0$.  To complete the proof of the necessity part, it remains to show that
\begin{equation}\label{equ-34a}
\int_{\abs{x}\le 1} \abs{x}^{\frac{1}{1-\gamma}}\,\rho(dx)<\infty.
\end{equation}
Since $\dot f(t)=\gamma t^{\gamma-1}$ for $t>0$, we have
\begin{equation}\label{basic-cal}
\int_0^\infty \big\{\abs{x \dot f(t)}\wedge \abs{x \dot f(t)}^2\big\}\,dt=C \abs{x}^{\frac{1}{1-\gamma}} 
\end{equation}
where $C=\gamma^{\frac{1}{1-\gamma}}(\gamma^{-1}+(1-2\gamma)^{-1})$. 
In the case  $\int_{\abs{x}\leq 1} \abs{x}\,\rho(dx)<\infty$ \eqref{equ-34a} holds    since $1<\tfrac{1}{1-\gamma}$. Thus we may assume that $\int_{\abs{x}\leq 1}\abs{x}\,\rho(dx)=\infty$, that is, \eqref{invar-con} of Theorem~\ref{thm-nes} is satisfied.  By 
Theorem~\ref{thm-nes}~\eqref{trunc_case} and \eqref{basic-cal} we have 
\begin{equation}
\int_{\abs{x}\leq 1} \abs{x}^{\frac{1}{1-\gamma}} \,\rho(dx)\leq \int_\R \abs{x}^{\frac{1}{1-\gamma}} (1\wedge x^{-2})\,\rho(dx)<\infty
\end{equation}
which completes the proof of the necessity part.   

On the other hand, suppose that $\sigma^2=0$, $\gamma\in (0,\tfrac{1}{2})$ and \eqref{int-finite-786} is satisfied. By \eqref{int-finite-786} and \eqref{basic-cal},  $f$ is absolutely continuous on $[0,\infty)$ with a derivative $\dot f$ satisfying \eqref{Cf} and hence $\bf X$ is a semimartingale with respect to $\Fi^Z$, cf.\ Theorem~\ref{thm-suf}. 
\end{proof}


\medskip

Below we will recall the conditions from  \cite{fv-mma} under which \eqref{eq:u0} or \eqref{eq:u00} hold.
Recall that a measure $\mu$ on $\R$ is said to be regularly varying if $x\mapsto \mu([-x,x]^c)$  is a regularly varying function; see \cite{Bingham}.  

\begin{proposition}[\citemma]\label{remark}  
Condition \eqref{eq:u0} is satisfied when one of the following two conditions holds for $m$-almost every $v \in V$
\begin{enumerate}[(i)]
	\item  \label{pro-con-1}  $\int_{\abs{x}>1} x^2 \, \rho_v(dx)<\infty$ or
\item  \label{pro-con-2} $\rho_v$ is regularly varying at $\infty$ with
          index $\beta\in [-2,-1)$. 
\end{enumerate}
Suppose that $\rho_v=\rho$ for all $v$, where $\rho$ satisfies \eqref{eq:u0} and is regularly varying with index $\bar \beta\in (-2,-1)$ at 0. Then  \eqref{eq:u00} holds.  
\end{proposition}

Theorems~\ref{thm-suf} and \ref{thm-nes} and Proposition~\ref{remark} extend  
\citet[Theorem~6.5]{Knight} from the  case where  
$\mathbf Z$ is a Brownian motion to quite general L\'evy processes in the following way. 

\begin{corollary}\label{cor-levy}
Suppose that   $\mathbf{Z}=(Z_t)_{t\in \R}$ is a two-sided  L\'evy process as above, with paths of  infinite variation on compact intervals. Let $\mathbf{X}=(X_t)_{t\geq 0}$ be a process of the form 
\begin{equation}\label{MA-case}
X_t=\int_{-\infty}^t \big\{f(t-s)-f_0(-s)\big\}\,dZ_s.
\end{equation}
Suppose that the random variable $Z_1$ is either square-integrable or has a regularly varying distribution at $\infty$ of index $\beta\in [-2,-1)$. 
Then $\bf X$ is a semimartingale with respect to $\Fi^Z$  if and only if $f$ is absolutely continuous on $[0,\infty)$ with a derivative $\dot f$ satisfying 
\begin{align}
& \int_0^\infty \abs{\dot f(t)}^2\,dt<\infty \qquad \text{if }\sigma^2>0, \\
& \label{levy-int-cor}\int_0^\infty \int_\R \big(\abs{x \dot f(t)}\wedge \abs{x \dot f(t)}^2\big)\,\rho(dx)\,dt<\infty.
\end{align}
\end{corollary}

\begin{proof}[Proof Corollary~\ref{cor-levy}]
The conditions imposed  on $Z_1$ are equivalent to that $\rho$ satisfies \eqref{pro-con-1} or \eqref{pro-con-2} of Proposition~\ref{remark}, respectively, cf.\  \cite[Theorem~1]{reg-var} and \cite[Theorem~25.3]{Sato}. Moreover, \eqref{invar-con} of Theorem~\ref{thm-nes} is equivalent to that $\bf Z$ has sample paths of infinite variation on bounded intervals and hence the result follows by Theorems~\ref{thm-suf} and \ref{thm-nes}. 
\end{proof}


\begin{example}\label{ex-stable}
 In the following we will consider  $\bf X$ and $\bf Z$  given as in Corollary~\ref{cor-levy} where  $\bf Z$ is either a stable or a tempered stable L\'evy process. 

\textbf{(i) Stable:}
Assume that  $\bf Z$ is a symmetric $\alpha$-stable L\'evy process with index $\alpha\in (1,2)$, that is, $\rho(dx)=c\abs{x}^{-\alpha-1}\,dx$ where $c>0$, and $\sigma^2=b=0$.  
Then $\bf X$ is a semimartingale with respect to $\Fi^Z$ if and only if $f$ is absolutely continuous on $[0,\infty)$ with a derivative $\dot f$ satisfying 
\begin{equation}\label{ex-stable-eq}
\int_0^\infty \abs{\dot f(t)}^\alpha\,dt<\infty.
\end{equation}
 We use  Corollary~\ref{cor-levy} to show the above. Note that  $\int_{\abs{x}\leq 1} \abs{x}\,\rho(dx)=\infty$ and $\rho$ is regularly varying at $\infty$ of index $-\alpha\in (-2,-1)$. Moreover, the identity
\begin{equation}\label{cal-stable}
\int_\R \big(\abs{xy}\wedge \abs{xy}^2\big)\,\rho(dx)=C \abs{y}^\alpha,\qquad y\in \R, 
\end{equation}
 with $C=2c ((2-\alpha)^{-1}+(\alpha-1)^{-1})$, shows that \eqref{levy-int-cor} is equivalent to \eqref{ex-stable-eq}. Thus the result follows by Corollary~\ref{cor-levy}.
 
 \textbf{(ii) Tempered stable:} Suppose  that $\bf Z$ is a symmetric tempered stable L\'evy process with indexs $\alpha\in [1,2)$ and $\lambda>0$, i.e., $\rho(dx)=c \abs{x}^{-\alpha-1}e^{-\lambda \abs{x}} \,dx$ where $c>0$,  and $\sigma^2=b=0$.  
Then $\bf X$ is a semimartingale with respect to $\Fi^Z$ if and only if $f$ is absolutely continuous on $[0,\infty)$ with a derivative $\dot f$ satisfying 
\begin{equation}\label{eq-temp-stable}
\int_0^\infty\big(\abs{ \dot f(t)}^{\alpha}\wedge \abs{ \dot f(t)}^2\big) \, ds <\infty.
\end{equation}
Again we will use Corollary~\ref{cor-levy}.  The conditions imposed on $\bf Z$ in Corollary \ref{cor-levy} are satisfied due to the fact that   $\int_{\abs{x}\leq 1} \abs{x}\,\rho(dx)=\infty$ and $\int_{\abs{x}>1}\abs{x}^2\,\rho(dx)<\infty$.  Moreover,  using the asymptotics of the incomplete gamma functions we have that 
\begin{equation}\label{tem-ex-levy23}
\int_\R \big(\abs{xu}\wedge \abs{xu}^2\big)\,\rho(dx)\sim \begin{cases} C_1 u^\alpha & 
\text{as } u\to \infty\\ C_2 u^2 & \text{as } u\to 0
\end{cases}
\end{equation}
where $C_1, C_2>0$ are finite constants depending only on $\alpha, c$ and $\lambda$, and we write $f(u)\sim g(u)$ as $u\to \infty$ (resp.\ $u\to 0$) when $f(u)/g(u)\to 1$ as $u\to \infty$ (resp.\ $u\to 0$).  Eq.~\eqref{tem-ex-levy23} shows that  \eqref{levy-int-cor} is equivalent to \eqref{eq-temp-stable}, and hence the result follows by Corollary~\ref{cor-levy}.   
\end{example}

\medskip
\begin{example}
A supOU process  ${\bf X}=(X_t)_{t\geq 0}$ is a stochastic process of  the  form 
\begin{equation}\label{eq-supOU}
X_t=\int_{\R_-\times (-\infty,t]} e^{v(t-s)}\,\Lambda(ds,dv)
\end{equation}
where $\R_-:=(-\infty,0)$, $\rho_v=\rho$ does not depend on $v$ and $m$ is a probability  measure. SupOU processes, which is short for  superposition of Ornstein--Uhlenbeck  processes, were   introduced by \citet{supOU}.  Suppose for simplicity that $\sigma^2=0$.  Process   $\bf X$ is well-defined if and only if  $\int_\R \log(1+|x|)\,\rho(dx)<\infty$ and $\int_{-\infty}^0 \frac{1}{|v|}\,m(dv)<\infty$, cf.\ \cite[page~343]{extremes-supOU}. 

Let $\bf X$ be a supOU process of the form \eqref{eq-supOU} and suppose that the L\'evy measure $\rho$ satisfies following (1)--(2):\begin{itemize}
\item[(1)] Either $\int_{|x|\geq 1} |x|^2\,\rho(dx)<\infty$, or $\rho$ is regularly varying at $\infty$ with index $\beta\in [-2,-1)$.
\item[(2)]  $\rho$ is regularly varying at $0$ with index $\bar\beta\in (-2,-1)$.
\end{itemize}
Then $\bf X$ is a semimartingale relative $\mathbbm F^\Lambda$ if and only if 
\begin{equation}\label{eq-yt}
\int_{-\infty}^0 \Big(\int_\R \big(|x v|^2\wedge |x v|\big) \,\rho(dx)\Big) |v|^{-1}\,m(dv)<\infty. 
\end{equation}
In particular if $\Lambda$ is symmetric $\alpha$-stable with $\alpha\in (1,2)$, i.e.\ $\rho(dx)=c |x|^{-1-\alpha}\,dx$, $c>0$.  Then $\bf X$ is a semimartingale with respect to $\mathbbm F^\Lambda$ if and only if 
\begin{equation}\label{stable-eq}
\int_{-\infty}^0 |v|^{\alpha-1}\, m(dv)<\infty. 
\end{equation}

To see this we observe that $f(t,v):=e^{vt}$ is absolutely continuous in $t\in [0,\infty)$ with $\dot f(t,v)=v e^{vt}$. For all $v\in \R_-$ and $x\in \R$ a simple computation shows  that 
\begin{equation}
\int_0^\infty |x \dot f(t,v)|\wedge |x\dot f(t,v)|^2\,dt
= \frac{|x v|^2}{2|v|}\1_{\{|xv|\leq 1\}} +\frac{|xv|-1/2}{|v|}\1_{\{|xv|>1\}}
\end{equation}
which is bounded from below and above by constants times 
\begin{equation}
\frac{1}{|v|} \Big(|x v|^2\wedge |x v|\Big).
\end{equation}
Thus \eqref{eq-yt} follows by Theorems~4.2 and 4.3 together with  Proposition~\ref{remark}. When $\Lambda$ is symmetric $\alpha$-stable with $\alpha\in (1,2)$,  the above (1) and (2) are satisfied and 
 $ \int_\R \big(|x v|^2\wedge |x v| \big)\,\rho(dx)=|v|^{\alpha}$. Hence \eqref{stable-eq} follows by \eqref{eq-yt}.
\end{example}

\medskip
\begin{example}[Multi-stable] \label{cor-stable}
In this example we  extend Example~\ref{ex-stable}(i) to the so called multi-stable processes, that is,  we will consider   $\bf X$  given by \eqref{eq:mma} with  
\begin{equation}
\rho_v(dx)=c \abs{x}^{-\alpha(v)-1} \, dx
\end{equation}
 where $\morf{\alpha}{V}{(0,2)}$ is a measurable function, $c>0$ and $b=\sigma^2=0$. For $v\in V$,  $\rho_v$ is the L\'evy measure of a symmetric stable distribution with index $\alpha(v)$.  Assume that 
 there exists an $r>1$ such that $\alpha(v)\geq r$ for all $v\in V$. Then  $\X$ is a semimartingale with respect to $\Fi^\Lambda$  if and only if  for $m$-a.e.\ $v$, $f(\cdot,v)$ is absolutely continuous on $[0,\infty)$ with  a derivative $\dot f(\cdot,v)$ satisfying
 \begin{equation}\label{stable-ex}
 \int_V\int_0^\infty\Big( \frac{1}{2-\alpha(v)}\abs{\dot f(s,v)}^{\alpha(v)}\Big) \,ds\,m(dv)<\infty.
 \end{equation}

To show the above  we will argue similarly as in Example~\ref{ex-stable}. By the symmetry, \eqref{drift-4} is satisfied. For all $v\in V$,  $\int_{\abs{x}\leq 1} \abs{x}\,\rho_v(dx)=\infty$, which shows that  \eqref{invar-con} of Theorem~\ref{thm-nes} is satisfied.  
By  basic calculus we have for $v\in V$ that 
\begin{equation}\label{int-stable-24}
u\int_{\abs{x}>u} \abs{x}\,\rho_v(dx)=K(v) \int_{\abs{x}\leq u} x^2\,\rho_v(dx)\end{equation}
where $K(v)=(2-\alpha(v))/(\alpha(v)-1)$. Since $\alpha(v)\geq r$ we have that 
$K(v)\leq 2/(r-1)<\infty$ which together with \eqref{int-stable-24} implies \eqref{eq:u00}. 
From  \eqref{cal-stable} we infer that \eqref{Cf} is equivalent to \eqref{stable-ex}, and thus Theorems~\ref{thm-suf} and \ref{thm-nes}  conclude the proof. 
   \end{example}
    \medskip

\begin{example}[supFLP]\label{cor:frac_levy}

 Consider  $\X=\p X$ of the form
\begin{equation}\label{def_multi_frac}
 X_t=\int_{\R\times V} \big( (t-s)_+^{\gamma(v)}-(-s)_+^{\gamma(v)}\big) \,\Lambda(ds,dv),
\end{equation}
where $\morf{\gamma}{V}{(0,\infty)}$ is a measurable function.
 Processes of the form \eqref{def_multi_frac} may be viewed as superpositions of fractional L\'evy processes with (possible) different indexes;  hence the name  supFLP.    If $m$-a.e.\ we have   $\gamma\in (0,\frac{1}{2})$, $\sigma^2 =0$ and 
  \begin{equation}\label{suf_con_frac}
 \int_V \Big(\int_\R \abs{x}^{\frac{1}{1-\gamma(v)}} \,\rho_v(dx)\Big) \big(\tfrac{1}{2}-\gamma(v)\big)^{-1}\,m(dv)<\infty,
\end{equation}
then  $\X$ is a semimartingale with respect to $\Fi^{\Lambda}$.
Conversely, if $\X$ is a semimartingale with respect to $\Fi^{\Lambda}$ and $\int_{\abs{x}\leq 1}\abs{x}\,\rho_v(dx)=\infty$ for $m$-a.e.\ $v$, then  $m$-a.e.\  $\gamma\in (0,\frac{1}{2})$, $ \sigma^2=0$ and 
\begin{equation}\label{eq:nece_ex}
 \int_\R \abs{x}^{\frac{1}{1-\gamma(v)}}\,\rho_v(dx)<\infty,
  \end{equation}
 and if in addition $\rho$ satisfies \eqref{eq:u00},  then \eqref{suf_con_frac} holds.     
  
  To show the above  let $f(t,v)=t^{\gamma(v)}_+$ for $t\in\R, v\in V$.
  Since   $f(0,v)=0$ for all $v$, \eqref{drift-4} is satisfied. As in Example~\ref{fLp}, we observe that the conditions
 \begin{equation}\label{eq-mult-32}
\int_{\abs{x}\geq 1} \abs{x}^{\frac{1}{1-\gamma(v)}}\,\rho_v(dx)<\infty\quad  \text{and} \quad \gamma(v)<\tfrac{1}{2}\quad  m\text{-a.e.}
 \end{equation}
 follow from the fact  that  $\X$ is a well-defined.  For $\gamma(v)\in (0,\frac{1}{2})$, $f(\cdot,v)$ is absolutely continuous on $[0,\infty)$. By \eqref{basic-cal}  we deduce that 
    \begin{align}\label{est-frac-1}
 \frac{c\abs{x}^{\frac{1}{1-\gamma(v)}}}{\tfrac{1}{2}-\gamma(v)}\leq \int_0^\infty \{\abs{x \dot f(t,v)}\wedge \abs{x \dot f(t,v)}^2\}\,dt \leq   
 \frac{\tilde c\abs{x}^{\frac{1}{1-\gamma(v)}}}{\tfrac{1}{2}-\gamma(v)}
   \end{align}
   for all $x\in \R$, 
where $c, \tilde c >0$ are finite constants not depending $v$ and $x$.

 By  Theorem~\ref{thm-suf} and \eqref{est-frac-1}, the sufficient part follows. To show the necessary part assume that $\X$ is a semimartingale with respect to $\Fi^{\Lambda}$ and that $\int_{\abs{x}\leq 1}\abs{x}\,\rho_v(dx)=\infty$ for $m$-a.e.\ $v$. 
  By Theorem~\ref{thm-nes},  $f(\cdot,v)$ is absolutely continuous with a derivative $\dot f(\cdot,v)$ satisfying \eqref{int_1} and \eqref{trunc_case}. From \eqref{int_1} we deduce that $\sigma^2=0$ $m$-a.e.\ and  from \eqref{trunc_case} and \eqref{est-frac-1} we infer that  
 \begin{equation}\label{eq-984}
\int_{\abs{x}\leq 1} \abs{x}^{\frac{1}{1-\gamma(v)}}\,\rho_v(dx)<\infty\quad m\text{-a.e.\ }v. 
 \end{equation}
    By   \eqref{eq-mult-32}--\eqref{eq-984}, condition \eqref{eq:nece_ex} follows. Moreover, if  $\rho$ satisfies \eqref{eq:u00},  then  Theorem~\ref{thm-nes} together with  \eqref{est-frac-1} show \eqref{suf_con_frac}. This completes the proof. 
 \end{example}

\appendix

\section{Appendix} \label{app}


In this appendix we will treat several of the results stated in the Introduction and Section~\ref{s-sem}. We have stated Stricker's theorem in a slightly extended version where it is combined with \cite[II, \textsection 4d]{Jacod_S}. In the following we treat   Examples~\ref{ex-1} and \ref{ex-2} in detail, discuss Theorem~\ref{thm-HC}, and prove some facts about the representation \eqref{def-X}. 

\medskip

\noindent
\textbf{Example~\ref{ex-1} (continued).}
Recall that $X_t=V+U$ for $t\in [0,1)$ and $X_t=V$ for $t\geq 1$ where $V$ is a Laplace distributed  random variable, that is, has a density $p_V(v)=(1/2)e^{-|v|}$, $v\in \R$, and $U$ is a standard Gaussian random variable independent of $V$. 
Process $\X$ is a special semimartingale with respect to $\mathbbm{F}^X=(\F^X_t)_{t\geq 0}$ with canonical decomposition $X_t=X_0+A_t+M_t$, where $A_t=0$ for $t<1$ and 
\begin{equation}
A_t= \E[\Delta X_1\,|\,\F_{1-}^X]  =-\E[U\,|\,U+V], \qquad  t\geq 1. 
\end{equation} 
Recall $\F^X_{1-}=\sigma(\cup_{s\in [0,1)} \F^X_s)$. 
The below Lemma~\ref{lem-not-id} shows that   $A_1$ is \emph{not} infinitely divisible. 

On the other hand, we may represent $\X$ by a random measure as 
\begin{equation}
X_t=\int_{(-\infty,t]\times V} \phi(t,u)\,\Lambda(du)
\end{equation}
where $V=\{1,2\}$ and $\Lambda$ is the random measure on $\R\times V$ such that for all $A\in \B(\R)$ and $B\subseteq \{1,2\}$ 
\begin{equation}
\Lambda\big(A\times B\big)=\delta_{0}(A)(\delta_{1}(B)V+\delta_{2}(B) U).
\end{equation}
Process $\X$ is  a special semimartingale with respect to $\Fi^\Lambda$ with canonical decomposition $X_t=X_0+A_t+M_t$ where $A_t=X_t-X_0$ and $M_t=0$. In particular, the processes $\M$ and $\A$ in the canonical decomposition of $\X$ in the filtration $\Fi^\Lambda$ are infinitely divisible.

In the following we will give a direct proof for that $\X$ is  not strictly representable. Notice that $\F^X_t=\sigma(U+V)$ for $t < 1$ and $\F^X_t=\sigma(U,V)$ for $t \ge 1$. Suppose to the contrary that
\begin{equation}\label{}
X_t=\int_{(-\infty,t]\times V} \phi(t,u)\,\Lambda(du)
\end{equation}
and $\F^X_t=\F^{\Lambda}_t$, for every $t\geq 0$.  We have
\begin{align*} 
  U+V &= X_0 = \int_{(-\infty,0] \times V} \phi(0,u)\,\Lambda(du) = J_0  \quad  \text{and} 
  \\
  V &= X_1 = \int_{(-\infty,0]\times V} \phi(1,u)\,\Lambda(du) + \int_{(0,1] \times V} \phi(1,u)\,\Lambda(du) = J_1 + J_2, 
\end{align*}
where $J_2$ is independent of $\{J_0, J_1\}$.
Since $V$ does not have Gaussian component, so do  $J_1$ and $J_2$. Now $U=J_0-J_1-J_2$ is Gaussian, so that $J_2$ cannot have Poissonian component either. Thus $J_2$ is deterministic, implying that $V$ is $\F_0^{\Lambda}= \F_0^{X}$-measurable. Hence $V=f(V+U)$ a.s.\ for some Borel function $\morf{f}{\R}{\R}$. Conditioning on $V$ we infer that, except of a set of Lebesgue measure zero, $f$ equals to a constant, a contradiction.

\begin{lemma}\label{lem-not-id}
Let $U$ and $V$ be given as in Example~\ref{ex-1}.  Then $\E[U\,|\,U+V]$ is a bounded random variable and  therefore not infinitely divisible.  
\end{lemma}

\begin{proof}
Set  $Y=U+V$. In addition, set $\Phi(y)=(2\pi)^{-1/2}\int_{-\infty}^y e^{-x^2/2}\,dx$,   $c_1=(2\sqrt{2\pi})^{-1}$ and  $c_2=\sqrt{e}/2$. By a calculation we get 
\begin{align}
& \E[U\,|\,Y=y]=\int_{\R} u\, \frac{p_{U,Y}(u,y)}{p_Y(y)}\,du\\  & \quad = \frac{-c_1e^{-y^2/2}+c_2e^{-y} \Phi(y-1)+c_1e^{-y^2/2}-c_2e^{y}\big[1-\Phi(y+1)\big]}{ c_2 e^{-y} \Phi(y-1) +c_2 e^{y} \big[1-\Phi(y+1)]}
\end{align}
from which we deduce that $y\mapsto \E[U\,|\,Y=y]$ is locally bounded and 
\begin{equation}\label{eq-36}
\lim_{y\to \pm \infty} \E[U\,|\,Y=y]= \pm 1. 
\end{equation}
Hence  $y\mapsto \E[U\,|\,Y=y]$ is  bounded and    $\E[U\,|\,Y]$  is a bounded random variable. Since  $\E[U\,|\,Y]$ is non-deterministic we conclude that it is not infinitely divisible, see \cite[Corollary~24.4]{Sato}.
\end{proof}

\medskip
\noindent
\textbf{Example~\ref{ex-2} (continued):}
Set $B_t=B_1(t)$.  
For all  $0\le s_1< \dots< s_n=s <t$ and $u_1,\dots,u_n \in \R$,
\begin{align} \label{eq-cha-12}
 & \E[(X_t-X_s) e^{i \sum_{j=1}^n u_j X_{s_j}}] = \frac{1}{i} \frac{\partial}{\partial \theta} \E [e^{i \theta (X_t -X_s) + i \sum_{j=1}^n u_j X_{s_j}}]\Big|_{\theta=0} \\
 &\qquad  = \frac{1}{i} \frac{\partial}{\partial \theta} \exp \Big(\E[e^{i \theta (B_t -B_s) + i \sum_{j=1}^n u_j B_{s_j}}]-1\Big)\Big|_{\theta=0}  \\ \label{eq-cha-13}
 &\qquad  = \exp \Big(\E[e^{ i \sum_{j=1}^n u_j B_{s_j}}]-1\Big)
 \E[(B_t-B_s) e^{i \sum_{j=1}^n u_j B_{s_j}}]=0.
\end{align}
Eq.~\eqref{eq-cha-12}--\eqref{eq-cha-13} show that  $\E [X_t-X_s\, |\, \F^X_s]=0$, that is,  $\X$ is a martingale.  For contradiction suppose that $\X$ has independent increments which, in particular, implies that 
\begin{equation}\label{eq-345}
\E[e^{i\theta (X_2-X_1)+iuX_1}]=\E[e^{i\theta (X_2-X_1)}]\E[e^{iuX_1}]
\end{equation}
for all $\theta, u\in \R$. By  \eqref{eq-345} it follows that 
\begin{equation}
\E[e^{i \theta (B_2 -B_1) + i u B_{1}}]=\E[e^{i \theta (B_2 -B_1) }]+\E[e^{i u B_{1}}]-1,
\end{equation}
 and hence 
\begin{equation}\label{eq-cha-func-1}
e^{-\theta^2/2- u^2/2}=e^{-\theta^2/2}+e^{- u^2/2}-1.
\end{equation}
Letting $\theta,u\to \infty$ the right-hand side of \eqref{eq-cha-func-1} tends to $-1$ while the left-hand side is positive. 
Thus, $\X$ can  not have independent increments. 

\medskip

\begin{remark}\label{rem-HC}
Theorem~\ref{thm-HC} follows from \citet[Theorem~$4.1'$]{Hida} applied to the process $(\bar X_t)_{t\in \R}$ defined by $\bar X_t=X_t$ for $t\geq 0$ and $\bar X_t=0$ for $t<0$. Notice that  $N$ and $B_j$  in Theorem~\ref{thm-HC} are different from the corresponding terms given in \cite[Theorem~$4.1'$]{Hida}. Simply, we have added the continuous and discontinuous components together to get a simpler representation. \end{remark}

\medskip

\begin{proposition}\label{representable}
Let   ${\bf X}=(X_t)_{t\geq 0}$ be an infinitely divisible process which is either {\rm (a)}  symmetric and right-continuous in probability or {\rm (b)} mean zero and right-continuous in $L^1$.
Then $\bf X$ is representable, i.e., it can be written in the form \eqref{rep-X-62}.
\end{proposition}
\begin{proof}
By \cite[Theorem~4.11]{Rosinski_spec}, under assumptions (a) or (b) there exist  an infinitely divisible random measure $\bar \Lambda=(\bar \Lambda(A))_{A\in \V}$ on a countable generated measurable space $(V,\V)$  and deterministic functions $\bar \phi(t,v)$ such that for all $t\geq 0$
\begin{equation}\label{eq1}
X_t=\int_V \bar \phi(t,v)\,\bar\Lambda(dv)\qquad \text{a.s.}
\end{equation}
Moreover, under (a) $\bar\Lambda$ is symmetric and $\bar\Lambda$ has mean zero under (b). 
We extend now $\bar \Lambda$ to an infinitely divisible random measure $\Lambda$ on $\R\times V$ by $\Lambda(A\times B):=\delta_0(A)\bar \Lambda(B)$, $A\in \mathcal{B}(\R)$, $B\in \V$.   Then \eqref{rep-X-62} holds with $\phi(t,u)=\bar\phi(t,v)$,  $u=(s,v)\in\R \times V$, $t\geq 0$.
\end{proof}

\medskip
\noindent
\textbf{Proof of Remark~\ref{cadlag}:} \  
Recall that  $\X=(X_t)_{t\geq 0}$ is a semimartingale relative $\Fi^\Lambda$ given by \eqref{def-X}, where $\Lambda$ satisfies the non-deterministic assumption \eqref{eq-kappa-42}, i.e.,
\begin{equation}\label{eq-kappa-42.1-wer}
\kappa\big(u\in \R\times V\! : \sigma^2(u)=0,\, \rho_u(\R)=0
\big)=0.
\end{equation}
 Then there exists a c\`adl\`ag modification of $\phi$.
More precisely, under \eqref{eq-kappa-42.1-wer} there exists a   mapping $\morf{\tilde \phi}{\R_+\times (\R\times V)}{\R}$ such that for all $u$, $t\mapsto \tilde \phi(t,u)$ is c\`adl\`ag  and for all $t\geq 0$,  $\phi(t,\cdot)=\tilde \phi(t,\cdot)$  $\kappa$-a.e. In fact, for a c\`adl\`ag process $\X$ on the form \eqref{def-X} there exists a function $\morf{ \phi_1}{\R_+\times (\R\times V)}{\R}$ such that $\phi(\cdot,u)$ is c\`adl\`ag and for all $t\geq 0$, $\phi(t,u)=\phi_1(t,u)$  for $\kappa$-a.e.\ $u$ with $\rho_u(\R)>0$, by similar arguments  as in  \citet[Theorem~4.1 and p.~86]{Rosinski_sum_rep}. On the other hand, if  $\X$ is a symmetric Gaussian semimartingale  then  there exists a function $\morf{ \phi_2}{\R_+\times (\R\times V)}{\R}$ such that $\phi(\cdot,u)$ is c\`adl\`ag and for $t\geq 0$, $\phi(t,u)=\phi_2(t,u)$  for $\kappa$-a.e.\ $u$ with $\sigma^2(u)>0$, cf.\ \citet[Theorem~4.6]{Andreas1}.  Hence by the symmetrization argument  used in  Case 3 in the proof of Theorem~\ref{thm1}  we obtain the c\`adl\`ag modification from the above  two  cases. 
\qed

\section{Two lemmas}\label{two}

In  this appendix we collect two results which are more or less
 well-known, but for which  we have not been able to find a reference.

\begin{lemma}\label{lem-fixed-dis-234}
Let  ${\bf Y}=(Y_t)_{t\geq 0}$ be a c\`adl\`ag process  with independent increments with respect to some filtration $\Fi$ and let $J=\{t\geq 0: \P(\Delta Y_t\neq 0)>0\}$ be the set of fixed discontinuities of $\bf Y$. Then there exists totally inaccessible   stopping times $(\tau_k)_{k\in \N}$  such that $\{\Delta \Y\neq 0\}= (\Omega\times J)\cup(\cup_{k\in \N} [\tau_k])$ up to evanescent. 
\end{lemma}

\begin{proof}
When $\Y$ is continuous in probability (i.e.\  $J=\emptyset$), Lemma~\ref{lem-fixed-dis-234}   follows by \cite[II, 5.12 and I, 2.26]{Jacod_S}. The general case may be shown as follows: By the decomposition theorem of stopping times, see  \cite[I, 1.32 and I, 2.22]{Jacod_S}, it is enough to show that 
for any predictable stopping time $S$ we have 
\begin{equation}\label{eq-jump-S-12}
\P(\Delta Y_{S}\neq 0, S\in J^c)=0.
\end{equation}
By the independent increments of $\Y$ and \cite[II, 1.17]{Jacod_S} it follows that the  predictable support of the random set $\{\Delta \Y\neq 0\}$ is  $\Omega\times J$. 
For a given   predictable stopping time $S$ let $A=\{S\notin J\}$ and $S_A=S\1_{A}+\infty\1_{A^c}$. Then $S$ is $\F_{S-}$-measurable since it is a stopping time, cf. \cite[I, 1.14]{Jacod_S}. We have  $A\in \F_{S-}$ since $A^c=\cup_{t\in J} \{S=t\}$ and $J$ is coutable. Thus by \cite[I, 2.10]{Jacod_S}, $S_A$ is a predictable stopping time.  Moreover, $(\Omega\times J)\cap [S_A]=\emptyset$ which implies that $\{\Delta \Y\neq 0\}\cap [S_A]$ is evanescent, see \cite[I, 2.33]{Jacod_S}, which is equivalent to   \eqref{eq-jump-S-12}. 
\end{proof}

\begin{lemma}\label{lem-cad-mod}
Let $\X=(X_t)_{t\geq 0}$ be an infinitely divisible process with independent increments. If $\X$ is c\`adl\`ag in probability then $\X$ has a c\`adl\`ag modification. 
\end{lemma}

Using that the characteristic function of any infinitely divisible random variable is non-zero everywhere the proof of Lemma~\ref{lem-cad-mod}  follows the  lines of  the proof of  \citet[Theorem~15.1]{Kallenberg}.

\bigskip

\noindent 
{\bf Acknowledgment.} 
Jan Rosi\'nski's research was partially supported by a grant \#281440 from the Simons Foundation.

\bibliographystyle{chicago}

\end{document}